\documentclass[12pt]{amsart}

\usepackage{amsmath}
\usepackage{amssymb}
\usepackage{amsthm}

\usepackage{graphicx}
\usepackage{tikz}

\usepackage{amsmath}
\usepackage{amsthm}
\usepackage{amscd}

\newtheorem{theorem}{Theorem}[section]
\newtheorem{lemma}[theorem]{Lemma}
\newtheorem{proposition}[theorem]{Proposition}

\theoremstyle{definition}
\newtheorem{definition}[theorem]{Definition}
\newtheorem{remark}[theorem]{Remark}
\newtheorem{example}[theorem]{Example}

\numberwithin{equation}{section}

\newcommand{\dsps}{\displaystyle}

\newcommand{\calF}{\mathcal{F}}
\newcommand{\calJ}{\mathcal{J}}
\newcommand{\calO}{\mathcal{O}}

\newcommand{\bOmega}{\mathsf{\Omega}}

\newcommand{\AAA}{\mathbb{A}}
\newcommand{\CC}{\mathbb{C}}
\newcommand{\NN}{\mathbb{N}}
\newcommand{\PP}{\mathbb{P}}
\newcommand{\RR}{\mathbb{R}}

\newcommand{\Cp}{\CC_p}
\newcommand{\Cv}{\CC_v}
\newcommand{\PCv}{\PP^1(\Cv)}

\newcommand{\Berk}{\textup{an}}
\newcommand{\ABerk}{\AAA^1_{\Berk}}
\newcommand{\PBerk}{\PP^1_{\Berk}}

\newcommand{\DBerk}{D_{\Berk}}
\newcommand{\DbarBerk}{\overline{D}_{\Berk}}

\newcommand{\Dbar}{\overline{D}}

\newcommand{\FBerkf}{\calF_{\Berk,f}}
\newcommand{\FBerkg}{\calF_{\Berk,g}}
\newcommand{\JBerk}{\calJ_{\Berk}}
\newcommand{\JBerkf}{\calJ_{\Berk,f}}
\newcommand{\JBerkg}{\calJ_{\Berk,g}}

\newcommand{\nat}{\natural}

\DeclareMathOperator{\diam}{diam}
\DeclareMathOperator{\sphdiam}{sphdiam}

\DeclareMathOperator{\PGL}{PGL}
\DeclareMathOperator{\Rat}{Rat}
\DeclareMathOperator{\CP}{CP}

\newcommand{\HBerk}{\mathbb{H}^1_{\Berk}}

\newcommand{\dd}[1]{D(#1)} 
\newcommand{\ddb}[1]{\DBerk(#1)} 
\newcommand{\sk}[1]{\|#1\|} 
\newcommand{\sd}[2]{#1^{\nat}(#2)} 
\newcommand{\spd}[1]{\sphdiam(#1)} 

\newcommand{\F}[2]{#1(#2)} 

\newcommand{\rrho}[1]{\rho(#1)} 

\definecolor{green}{rgb}{.1,.4,.1}
\definecolor{blue}{rgb}{.2,.6,.75}

\begin{document}

\title[J-Stability]{J-Stability in non-archimedean dynamics}


\author{Robert L. Benedetto}
\address{Amherst College, Amherst, MA 01002, USA}
\email{rlbenedetto@amherst.edu}
\author{Junghun Lee}
\address{Department of Mathematics, Tokyo Institute of Technology, Tokyo, Japan}
\email{lee.j.ba@m.titech.ac.jp}



\date{October 8, 2021}

\begin{abstract}
Let $\Cv$ be a complete, algebraically closed non-archimedean field,
and let $f \in\Cv(z)$ be a rational function of degree $d \geq 2$.
If $f$ satisfies a bounded contraction condition on its Julia set, we
prove that small perturbations of $f$ have dynamics conjugate to those of $f$ on their Julia sets.
\end{abstract}

\maketitle


\section{Introduction}

Fix the following notation throughout this paper.

\begin{tabbing}
\hspace{5mm} \= \hspace{17mm} \= \kill  
\> $\Cv$ \> an algebraically closed field of characteristic zero. \\
\> $|\cdot|$ \> a nontrivial non-archimedean absolute value on $\Cv$, \\ 
\> \> with respect to which $\Cv$ is complete. \\
\> $\NN$ \> the set $\{1,2,3,\ldots\}$ of positive integers. \\
\> $\NN_0$ \> $\NN\cup\{0\}$.
\end{tabbing}

The Berkovich projective line $\PBerk$ is a natural compactification
of the classical projective line $\PCv=\Cv\cup\{\infty\}$, which we describe
in greater detail in Section~\ref{sec:berksp}.
We consider the dynamics of a rational function $f\in\Cv(z)$ on $\PCv$ and on $\PBerk$.
That is, writing $f^0(z)=z$ and $f^{n+1}=f\circ f^n$ for all $n\in\NN_0$, we consider the action of the
iterates $f^n$ on $\PCv$ and $\PBerk$.
See \cite[Chapter 10]{BR10}, \cite{BenBook}, or \cite[Chapter 5]{Silv07}
for more thorough treatments of such non-archimedean dynamics.
We will be especially interested in the case that two such maps $f,g\in\Cv(z)$ are 
conjugate on a subset of $\PBerk$;
more precisely, there is some invertible map $h:V\to V$
such that $h\circ f|_W = g\circ h|_W$, where
$W=f^{-1}(V)\subseteq V\subseteq\PBerk$.

A rational function $f\in\Cv(z)$ may be written as $f=F/G$ for relatively prime polynomials $F,G\in\Cv[z]$.
We define the degree of $f$ to be $\deg f := \max\{\deg F, \deg G\}$. Every point of $\PCv$ has $\deg f$ preimages
under $f$, counted with multiplicity. For any integer $d\geq 2$, we define
\[ \Rat_d(\Cv) := \big\{ f\in \Cv(z) \, \big| \, \deg f =d \big\} \]
to be the set of rational functions of degree $d$, defined over $\Cv$, with the topology induced from
the natural inclusion of $\Rat_d(\Cv)$ in $\PP^{2d+1}(\Cv)$, which maps $f$ to the $(2d+2)$-tuple of its coefficients.

The main result of this paper is motivated by Ma\~n\'e, Sad, and Sullivan's result \cite{MSS83} in complex dynamics. 
They introduced the notion of $J$-stability of a rational map $f\in\CC(z)$,
a property which, roughly speaking, means that the dynamics of all maps $g$ in some neighborhood
of $f$ in $\Rat_d(\CC)$ are conjugate on their Julia sets.
In particular, they showed that a rational map is $J$-stable if it
is expanding on its Julia set.
For more discussion of stability in the complex setting, see also
\cite[Section~4.1]{McMull94} and \cite[Sections~7--8]{MS98}. 

Motivated by the results of \cite{MSS83}, T.~Silverman \cite{TSil17}
proved a non-archimedean stability result for one-parameter families
via a condition on the Berkovich analytification of the appropriate moduli space.
In a different direction, the second author \cite{L18} investigated non-archimedean
rational functions $f\in\Cv(z)$ acting on $\PCv$, proving that $f$ is $J$-stable if it is
expanding in a sense parallel to that in complex dynamics.
Specifically, as in \cite[Definition~1.1]{L18}, the map $f\in\Cv$ is
\emph{expanding} on its (type~I) Julia set $J_f:=\JBerkf\cap\PCv$
if $J_f$ is nonempty and there exist real constants $c>0$ and $\lambda>1$ so that
\begin{equation}
\label{eq:expand}
\big(f^n\big)^{\nat}(z) \geq c\lambda^n
\quad\text{for every } z\in \JBerkf\cap\PCv  \text{ and } n\in\NN, 
\end{equation}
where $g^{\nat}$ denotes the spherical derivative of $g\in\Cv(z)$, defined in Section~\ref{sec:prelimSD}.
(See also Remark~\ref{rem:expand}.)
In a different context \cite{B01b}, the first author had previously studied a slightly weaker version of this condition
for the case that $\Cv=\Cp$ and $f$ is defined over a locally compact subfield $K$ of $\Cv$.
(Specifically, such a map $f$ is \emph{hyperbolic} if for each finite extension $L/K$, there exist
$c=c_L>0$ and $\lambda=\lambda_L>1$ such that condition~\eqref{eq:expand} holds for all $z\in J_f\cap\PP^1(L)$.)
However, besides the fact that the results of \cite{L18} apply only to the type~I Julia set $J_f$,
both the expanding and the hyperbolic hypotheses are unnecessarily restrictive, as we illustrate in Section~\ref{sec:ex}.


In this paper, we strengthen the main result of \cite{L18} both by generalizing the expanding hypothesis
of equation~\eqref{eq:expand}
and by extending the resulting conjugacy from the classical Julia set in $\PCv$ to the Berkovich Julia set
$\JBerkf\subseteq\PBerk$ of the map $f$. (See Section~\ref{sec:berkdyn} for more on the Berkovich Julia set.)
Moreover, we construct our conjugacy not only on $\JBerkf$, but also 
on an appropriate neighborhood of $\JBerkf \cap \PCv$ in $\PBerk$. 
As in~\cite{L18}, our statement involves the spherical derivative $f^{\nat}$ of the rational function $f$,
but extended to the Berkovich space $\PBerk$, as described in Section~\ref{sec:prelimSD},
and with a less restrictive hypothesis. Our extension to $\PBerk$ also allows us to avoid the assumption
that $\JBerkf\cap\PCv\neq\varnothing$ required in both \cite{L18} and \cite{TSil17}.
On the other hand, although we prove that our conjugacy
varies continuously with the map $g\in\Rat_d(\Cv)$,
our method does not yield analytic motions
of Julia sets as in \cite{TSil17}, in part because
we do not consider nonclassical Berkovich points in
the moduli space $\Rat_d$. 

\begin{theorem}\label{MAIN}
Let $f\in\Cv(z)$ be a rational function of degree $d\geq 2$
with Berkovich Julia set $\JBerkf$.
Suppose there exists $\delta >0$ such that
\[ \big(f^n \big)^{\nat}(\zeta) \geq \delta \quad \text{for all } \zeta\in\JBerkf
\text{ and } n\in\NN.  \]
Then there exist a neighborhood $W \subseteq \Rat_d(\Cv)$ of $f$
and an open set $U\subseteq\PBerk$ containing $\JBerkf\cap\PCv$
with the following properties.
For each $g \in W$, there is a homeomorphism
$h:\PBerk\to\PBerk$
for which
\begin{enumerate}
\item $h$ is an isometry on the set $\PCv$ of type~I points ,
\item $h$ is the identity map on $\PBerk \smallsetminus U$, and
\item $h\circ f(\zeta) = g\circ h(\zeta)$
for all $\zeta\in U\cup\JBerkf$.
\end{enumerate}
Moreover, the map $(g,\zeta)\mapsto h(\zeta)$
is a continuous function from $W\times\PBerk$ to $\PBerk$. 
\end{theorem}

Note in particular that the points of $\JBerkf\smallsetminus U$
are fixed by the map $h$ of Theorem~\ref{MAIN}.
Hence, we have $\JBerkf\smallsetminus U = \JBerkg\smallsetminus U$,
and moreover
$f(\zeta)=g(\zeta)$ for all $\zeta\in\JBerkf\smallsetminus U$. 



The outline of this paper is as follows.
We recall some essentials from non-archimedean analysis and dynamics in Section~\ref{sec:prelim}, and we describe the spherical derivative on $\PBerk$ in Section~\ref{sec:prelimSD}.
Next, we present several necessary lemmas in Sections~\ref{sec:somelemmas} and~\ref{sec:proofs}.
Section~\ref{sec:main} is devoted to the proof of Theorem~\ref{MAIN}.
Finally, in Section \ref{sec:ex}, we present examples of rational maps which satisfy
the hypotheses of Theorem~\ref{MAIN} but which are not expanding in the sense of~\cite{L18}.

\section{Preliminaries}
\label{sec:prelim}

In this section, we recall some relevant facts about dynamics on $\PCv$ and on $\PBerk$.
Here and in the rest of the paper, we set the following notation for disks in $\Cv$.

\begin{tabbing}
\hspace{5mm} \= \hspace{17mm} \= \kill  
\> $D(a, r)$ \> for $a\in \Cv$ and $r>0$, the open disk $\{ x \in \Cv \mid |x - a| < r \}$. \\
\> $\Dbar(a, r)$ \> for $a\in \Cv$ and $r>0$, the closed disk $\{ x \in \Cv \mid |x - a| \leq r \}$. \\
\> $\calO$ \> the ring of integers $\Dbar(0,1)=\{ z \in \Cv \mid |z| \leq 1 \}$ of $\Cv$.
\end{tabbing}

\subsection{The chordal metric}
\label{sec:prelimPL}


The \emph{chordal metric} is the distance function $\rho$ on $\PCv$ given in homogeneous coordinates by
\[ \rho\big( [z_0: z_1], [w_0: w_1] \big) :=
\frac{ |z_0 w_1 - z_1 w_0 | }{\max\{ |z_0|, |w_0|\} \max\{ |z_1|, |w_1| \} }. \]
Equivalently, in affine coordinates we have
\begin{align*}
\rho\big( z, w \big) = 
\frac{ |z - w| }{\max\{ 1, |z| \} \max\{ 1, |w| \} } =
\begin{cases}
| z - w | &\quad \text{if  }z, w \in \calO, \\[2mm]
\dsps \left| \frac{1}{z} - \frac{1}{w} \right| &\quad \text{if } z, w \in \Cv\smallsetminus\calO, \\[2mm]
1 &\quad \text{otherwise.}
\end{cases}
\end{align*}
Any $h\in\PGL(2,\calO)$ is an isometry with respect to the chordal metric.
See \cite[Section~2.1]{Silv07} or \cite[Section~5.1]{BenBook} for more on the chordal metric.

\subsection{Weierstrass degrees of power series}
\label{sec:prelimw}

Let $a\in \Cv$ and $r>0$. A power series
\[ F(z) = \sum_{i=0}^{\infty} c_i (z-a)^i \in \Cv[[z-a]] \]
converges on $D(a,r)$ if and only if 
\[ \lim_{n \to \infty} |c_n| s^n = 0 \quad \text{for all $0 < s < r$.} \]
If $F$ converges on $D(a, r)$, then the derivative of $F$
\[ F'(z) = \sum_{i=1}^{\infty} i c_i (z-a)^{i - 1} \in \Cv[[z-a]] \]
also converges on $D(a,r)$. In particular, $F(a) = c_0$ and $F'(a) = c_1$.

The \emph{Weierstrass degree} of $F$ on $D(a,r)$ is defined to be the smallest $n\in\NN_0$ such that
\[|c_n|r^n = \sup\{ |c_i|r^i \mid i \in \mathbb{N}_0  \}, \]
or $\infty$ if this supremum is never attained.
If $n \in \NN$ is the Weierstrass degree of $F-c_0$ on $D(a, r)$,
then $F$ maps $D(a,r)$ onto the disk $D(c_0, |c_n|r^n )$,
and every point of the latter disk has $n$ preimages in the former, counted with multiplicity.
In particular, $F$ is injective on $D(a,r)$ if and only if $n = 1$, in which case $F(D(a,r)) = D(F(a), |c_1| r)$, and 
\[ |F(x)-F(y)|=|F'(a)| |x-y| \quad \text{for all $x,y \in D(a,r)$}. \]
If $F$ is injective on $D(a, r)$, then $F'$ has no zeros in $D(a,r)$.
However, the converse is not necessarily true if $\Cv$ has positive residue characteristic,
although Lemma~\ref{lem:kappa} shows that $F$ is injective on a smaller disk in that case.

If $f\in \Rat_d(\Cv)$ has no poles in $D(a,r)$, then there exists a convergent power series
$F \in \Cv[[z - a]]$ on $D(a,r)$ such that $F(x) = f(x)$ for all $x \in D(a, r)$.
Thus, the image $f(D(a,r))$ is a disk of the form $D(b,s)$, where $b=f(a)$.
Note that the Weierstrass degree of $F-b$ on $D(a, r)$ is at most $d$.

We refer the reader to \cite[Chapters~3,14]{BenBook} or \cite[Chapter~6]{Rob00} for more details on power series
over non-archimedean fields.

\subsection{The Berkovich projective line}
\label{sec:berksp}

It has become clear that although a significant amount of non-archimedean dynamics can be done
on the classical projective line $\PCv$, the appropriate setting is the Berkovich projective line $\PBerk$.
In this section we summarize some relevant facts about $\PBerk$ and its associated dynamics.
For more details, see \cite[Chapters~1,2,9,10]{BR10} or \cite[Chapters~6--8]{BenBook}.

The Berkovich affine line $\ABerk$ is the set of all multiplicative seminorms on $\Cv[z]$
that extend the absolute value on $\Cv$. That is, $\zeta=\|\cdot\|_{\zeta}$ is a function from
$\Cv[z]$ to $[0,\infty)$ satisfying $\|fg\|_{\zeta}=\|f\|_{\zeta}\|g\|_{\zeta}$,
$\|f+g\|_{\zeta}\leq\max\{\|f\|_{\zeta},\|g\|_{\zeta}\}$, and $\|a\|_{\zeta}=|a|$
for all $f,g\in\Cv[z]$ and $a\in\Cv$.
We will generally write an element of $\ABerk$ as $\zeta$ when we think of it as a point,
and as $\|\cdot\|_{\zeta}$ when we think of it as a seminorm.

There are four types of points in $\ABerk$. Type~I points correspond to the points of $\Cv$,
with $\|f\|_x:=|f(x)|$ for $x\in\Cv$. Points of type~II and~III correspond to closed disks $\Dbar(a,r)$,
with $a\in\Cv$ and $r>0$, where $r\in |\Cv^{\times}|$ gives a point of type~II, and
$r\not\in |\Cv^{\times}|$ gives a point of type~III. In both cases, the corresponding point
$\zeta(a,r)\in\ABerk$ is the sup-norm on the disk $\Dbar(a,r)$.
Finally, type~IV points correspond to descending chains of disks
$D_1\supsetneq D_2\supsetneq \cdots$ with empty intersection.
We denote by $\HBerk:=\ABerk\smallsetminus\Cv$ the subset of points
not of type~I.

We equip $\ABerk$ with the Gel'fand topology, i.e.,
the weakest topology such that for every $f\in\Cv[z]$, the function $\zeta\mapsto\| f\|_{\zeta}$
maps $\ABerk$ continuously to $\RR$. The projective line $\PBerk$ may be formed either by
taking the one-point compactification $\PBerk=\ABerk\cup\{\infty\}$ or by gluing two copies 
of $\ABerk$ via $\zeta\mapsto 1/\zeta$. (The new point $\infty$ is of type~I.)
Then $\PBerk$ is a compact Hausdorff space which contains $\PCv$, the set of type~I points,
as a dense subspace.

For $a\in\Cv$ and $r>0$, the sets
\[ \DBerk(a,r) := \{ \zeta\in\ABerk \, | \, \|z-a\|_{\zeta} < r \}
\quad\text{and}\quad
\DbarBerk(a,r) := \{ \zeta\in\ABerk \, | \, \|z-a\|_{\zeta} \leq r \} \]
are called open and closed Berkovich disks, respectively.
A type~I point $x\in\Cv$ lies in $\DbarBerk(a,r)$ if and only if $x\in \Dbar(a,r)$,
and it lies in $\DBerk(a,r)$ if and only if $x\in D(a,r)$.
A type~II or~III point $\zeta=\zeta(b,s)$ lies in $\DbarBerk(a,r)$
if and only if $\Dbar(b,s)\subseteq\Dbar(a,r)$;
and it lies in $\DBerk(a,r)$ if and only if $\Dbar(b,s)\subseteq D(a,r)$.
(The one exception to the last rule is that a type~III point $\zeta(a,r)$ itself does not lie in $\DBerk(a,r)$,
even though $\Dbar(a,r)=D(a,r)$ for $r\not\in |\Cv^{\times}|$.)
As is the case for disks in $\Cv$, if two Berkovich disks intersect, then one disk contains the other.

If $\zeta$ lies in the Berkovich disk $\DBerk(a,r)$, we will sometimes abuse notation
and write $\DBerk(a,r)$ as $\DBerk(\zeta,r)$, and even $D(a,r)$ as $D(\zeta,r)$.
We will similarly write $\DbarBerk(a,r)=\DbarBerk(\zeta,r)$ and $\Dbar(a,r)=\Dbar(\zeta,r)$
if $\zeta\in\DbarBerk(a,r)$.

More generally, an \emph{open connected affinoid} is either $\PBerk$ with finitely many closed Berkovich disks
removed, or else an open Berkovich disk with finitely many closed Berkovich disks removed.
A \emph{closed connected affinoid} is defined similarly, with the roles of ``open'' and ``closed'' reversed.
The open connected affinoids form a basis for the Gel'fand topology on $\PBerk$.
If $U$ is either an open or closed connected affinoid, then both the set of type~I points of $U$
and the set of type~II points of $U$ are dense in $U$.

\subsection{Dynamics on the Berkovich line}
\label{sec:berkdyn}

Any seminorm $\zeta\in\ABerk$ extends from $\Cv[z]$ to $\Cv(z)$ by defining
$\|F/G\|_\zeta := \|F\|_{\zeta} / \|G\|_{\zeta}$, where we understand $\infty$ to be a legal
value for this expression, in case $\|G\|_{\zeta}=0$. Any rational function $f\in\Cv(z)$ then defines
a continuous function $f:\PBerk\to\PBerk$, given by
\[ \|h\|_{f(\zeta)} := \| h\circ f\|_{\zeta} , \]
which extends the usual action of $f$ on the type~I points of $\PCv$.

Moreover, if $f$ is a convergent power series on $D(a,r)$, then $f$ similarly induces a continuous
function $f:\DBerk(a,r)\to \ABerk$.
For any open disks $D(a,r), D(b,s)\subseteq\Cv$, we have
\[ f\big( D(a,r) \big) = D(b,s) \Longleftrightarrow
f\big( \DBerk(a,r) \big) = \DBerk(b,s) . \]
Furthermore, in that case, 
the following are equivalent:
\begin{itemize}
\item $f(z)-b$ has Weierstrass degree $1$ on $D(a,r)$.
\item $f:D(a,r)\to D(b,s)$ is a bijective function.
\item $f:\DBerk(a,r)\to \DBerk(b,s)$ is a bijective function.
\item $f$ has an inverse function $f^{-1}: D(b,s) \to D(a,r)$ also given by a convergent power series.
\end{itemize}
(The fact that bijectivity implies Weierstrass degree $1$ uses our assumption that $\Cv$ has characteristic zero;
that implication fails in positive characteristic for totally inseparable maps.)

The (Berkovich) \emph{Fatou set} of a rational function $f\in\Cv(z)$ of degree $d\geq 2$ is
the set of points $\zeta\in\PBerk$ having a neighborhood $U$ such that $\bigcup_{n\in\NN} f^n(U)$
omits infinitely many points of $\PBerk$. The complement $\PBerk\smallsetminus\FBerkf$
is the (Berkovich) \emph{Julia set} $\JBerkf$ of $f$.
Both sets are nonempty (see \cite[Corollaries~5.15 and~12.6]{BenBook}),
and both are invariant under $f$, meaning that
\[ f^{-1}(\JBerkf)=f(\JBerkf)=\JBerkf \quad\text{and}\quad f^{-1}(\FBerkf)=f(\FBerkf)=\FBerkf. \]
The Fatou set is open in $\PBerk$, and the Julia set is closed (and hence compact).

The type~I Fatou set $\FBerkf\cap\PCv$ consists of those points of $\PCv$ having a neighborhood
on which the sequence of iterates $\{f^n\}_{n=0}^{\infty}$ is equicontinuous with respect to the chordal metric $\rho$.
If a type~I point $x\in\Cv$ is periodic, i.e., if $f^n(x)=x$ for some
(minimal) positive integer $n\in\NN$, then the \emph{multiplier} of $x$ is $(f^n)'(x)$.
If the multiplier $\lambda$ of $x$ satisfies $|\lambda|>1$, then $x$ is said to be \emph{repelling},
and we have $x\in\JBerkf$. Otherwise, i.e. if $|\lambda|\leq 1$, then $x$ is said to be \emph{nonrepelling},
and $x\in\FBerkf$.

\section{The spherical derivative} 
\label{sec:prelimSD}

The spherical kernel is a natural extension to $\PBerk$
of the chordal metric $\rho$ on $\PCv$.
%
%
We recall its definition and some of its properties from
\cite[Section~4.3]{BR10}.

\begin{definition}
\label{def:sphker}
    The \emph{spherical kernel} is the unique function $\sk{\cdot, \cdot} : \PBerk \times \PBerk \to \RR$ such that
    \begin{itemize}
        \item $\sk{x,y} = \rrho{x,y}$ for any $x,y \in \PCv$,
        \item $\sk{\cdot, \cdot}$ is continuous on
        $\PBerk \times \PBerk \smallsetminus
        \{ (\zeta, \zeta) \mid \zeta \in \HBerk \}$,
     and
        \item for any $\zeta, \xi \in \PBerk$,
        \[ 
            \sk{\zeta, \xi} = \limsup_{(x,y)\to(\zeta,\xi)} \rrho{x,y}
        \]
        where the $\limsup$ is over $(x,y)\in\PCv\times\PCv$.
    \end{itemize}
\end{definition}
See \cite[Equation~(4.21)]{BR10} for an explicit construction of
the spherical kernel.
Although it is not a metric,
the spherical kernel has the following related properties;
see \cite[Proposition~4.7]{BR10}.

\begin{proposition}
\label{prop:BR4.7}
    The spherical kernel is symmetric and takes values in $[0,1]$. Moreover, it is
    continuous in each variable separately, and it is upper semicontinuous
    as a function of two variables.
\end{proposition}

The spherical kernel is discontinuous on the diagonal in $\HBerk\times\HBerk$,
but it is precisely on this diagonal that we are most interested in it,
as illustrated by the next two definitions.

\begin{definition}
    The \emph{spherical diameter} $\spd{\cdot}: \PBerk \to [0,1]$ is defined by 
    \[
        \spd{\zeta} := \sk{\zeta, \zeta}
    \]
    for any $\zeta \in \PBerk$.
\end{definition}

In \cite[Section~6.1.2]{BenBook}, the diameter of $\zeta\in\ABerk$ is defined as 
\[ 
    \diam(\zeta)
    :=\inf\{\|z-a\|_{\zeta} : a\in\Cv\}
\]
Defining $|\zeta|:=\|z\|_{\zeta}$, we have the identity
\[ 
    \sphdiam(\zeta) 
    =
    \frac{\diam(\zeta)}{\max\{1,|\zeta|^2\} }
    \quad\text{for all } \zeta\in\ABerk,
\]
with $\sphdiam(\infty) = 0$.



\begin{definition}
    Let $f\in\Cv(z)$ be a rational function.
    The \emph{spherical derivative} of $f$ on $\PBerk$ is
    \[ 
        \sd{f}{\zeta} 
        := \lim_{\zeta' \rightarrow \zeta} \frac{\big\|f(\zeta), f(\zeta') \big\|}{\|\zeta, \zeta' \| }
    \] 
    where the convergence $\zeta' \to \zeta$ is with respect to
    the Gel'fand topology on $\PBerk$.
\end{definition}

Our next result shows how to compute the spherical derivative in practice.

\begin{proposition}
\label{prop:sphcomp}
    Let $f\in\Cv(z)$ be a nonconstant rational function, and let $\zeta\in\PBerk$.
    \begin{enumerate}
        \item If $\zeta=x\in\PCv$, then $f^{\nat}(x)=f^{\#}(x)$ is the
        classical spherical derivative on $\PCv$, given by
        \[ f^{\nat}(x)=f^{\#}(x):= \lim_{y\to x} \frac{\rho\big(f(x),f(y)\big)}{\rho(x,y)}
        \in\RR_{\geq 0},\]
        where $y\to x$ in $\PCv$. In particular, if $x,f(x)\in\Cv$, then
        \[ f^{\nat}(x)=|f'(x)| \cdot \frac{\max\{1,|x|^2\}}{\max\{1,|f(x)|^2\}} .\]
        \item If $\zeta\in\HBerk$, then
    
    \[
        \sd{f}{\zeta} 
        = \frac{\sphdiam\big(f(\zeta)\big)}{\sphdiam(\zeta)} \in \RR_{>0}.
    \]
    \end{enumerate}
    \end{proposition}

\begin{proof}
    By Proposition~\ref{prop:BR4.7}, the maps
    $\sk{\zeta, \cdot} : \PBerk \to \RR$ and $\sk{\F{f}{\zeta}, \cdot} : \PBerk \to \RR$
    are continuous. Since $f : \PBerk \to \PBerk$ is also continuous, we have
    \[
        \lim_{\zeta' \to \zeta} \big\|\F{f}{\zeta}, f(\zeta')\big\|
        = \big\|\F{f}{\zeta}, \F{f}{\zeta}\big\|
        = \sphdiam\big( f(\zeta) \big),
    \]
    and
    \[
        \lim_{\zeta' \to \zeta} \sk{\zeta, \zeta'} 
        = \sk{\zeta, \zeta} 
        = \spd{\zeta}.
    \]

    If $\zeta \in \HBerk$, then $\spd{\zeta}>0$.
    Because $f$ is nonconstant, we have $\F{f}{\zeta} \in \HBerk$ as well,
    and hence also $\spd{\F{f}{\zeta}} > 0$. Therefore,
    \[
        \sd{f}{\zeta} 
        = \frac{\sphdiam\big(f(\zeta)\big)}{\spd{\zeta}} > 0.
    \]

    Otherwise, we have $\zeta=x\in\PCv$,
    so that $\F{f}{\zeta}=f(x) \in \PCv$, and hence
    \[ 
    \sd{f}{\zeta} 
            = \lim_{\zeta' \rightarrow \zeta} \frac{\big\|f(\zeta), f(\zeta') \big\|}{\|\zeta, \zeta'\|}
            = \lim_{y\to x}
            \frac{\big\|f(x), f(y) \big\|}{\|x,y\|} \\
            = \lim_{y\to x} \frac{\rho\big(f(x), f(y)\big)}{ \rrho{x,y}}
            = f^{\#}(x) \geq 0,
    \]
    where the second and third limits are for $y\to x=\zeta$ in $\PCv$,
    and the second equality follows from the density of $\PCv$ in $\PBerk$.
\end{proof}

Recall that the chordal metric $\rho$
is invariant under the action of $\PGL(2, \calO)$.
Therefore, by the third bullet point of Definition~\ref{def:sphker},
we have
\begin{equation}
\label{eq:natinvar}
h^{\nat}(\zeta)=1 \quad \text{for all } h \in \PGL(2, \calO)
\text{ and } \zeta\in\PBerk .
\end{equation}

The spherical derivative also satisfies the following chain rule.

\begin{proposition}\label{prop:chainrule}
    For any rational functions $f, g\in\Cv(z)$ and any $\zeta\in\PBerk$, we have
    \[
        \sd{(f \circ g)}{\zeta} 
        = f^{\nat}\big(g(\zeta)\big) \cdot g^{\nat}(\zeta)
    \]
\end{proposition}

\begin{proof}
    By continuity, we have
    \begin{align*}
        \sd{(f \circ g)}{\zeta}
        &= \lim_{\zeta' \to \zeta} \frac{ \big\| f\big(g(\zeta)\big), f\big(g(\zeta')\big)\big\|}{\|\zeta, \zeta'\|}
        = \lim_{\zeta' \to \zeta} \frac{ \big\| f\big(g(\zeta)\big), f\big(g(\zeta')\big)\big\|}{\big\|g(\zeta), g(\zeta') \big\|}
        \cdot \lim_{\zeta' \to \zeta} \frac{\big\|g(\zeta), g(\zeta') \big\|}{\|\zeta, \zeta'\|}
        \\
        &= f^{\nat}\big(g(\zeta)\big) \cdot \sd{g}{\zeta}
        \qedhere
    \end{align*}
\end{proof}

We close this section with the following lemma concerning a disk
on which the rational function $f$ has Weierstrass degree $1$.

\begin{lemma}\label{lem:injdiam}
    Let $f \in \Cv(z)$, let $a, b \in \Cv$ with $|a|,|b| \leq 1$, and $0 < r, s \le 1$.
    Suppose $f$ maps $\dd{a, r}$ bijectively onto $\dd{b, s}$. Then $\sd{f}{\zeta} = s/r$ for any $\zeta \in \ddb{a, r}$.
    In particular, we have
    \[
        \sphdiam\big(f(\zeta)\big)
        = \frac{s}{r} \spd{\zeta}
    \]
\end{lemma}

\begin{proof}
By \cite[Proposition~3.20]{BenBook}, we have
\begin{equation}
\label{eq:xysize}
\big|f(x)-f(y)\big| = \frac{s}{r} |x-y| \quad\text{for all } x,y\in D(a,r) .
\end{equation}
Recall that $\rrho{x, y} = |x - y|$ for any $x,y \in \calO$.
Because $D(a,r),D(b,s)\subseteq\calO$, it follows that
\[ \big\| f(x), f(y) \big\| = \frac{s}{r} \|x,y\| \quad\text{for all } x,y\in D(a,r) . \]
Therefore, by the third bullet point of Definition~\ref{def:sphker}, we have
\[ \big\| f(\zeta), f(\zeta) \big\| = \frac{s}{r} \| \zeta, \zeta \|
\quad\text{for all } \zeta\in\DBerk(a,r), \]
which is the desired conclusion for $\zeta$ not of type~I,
by definition of the spherical diameter.
Finally, the conclusion for $\zeta$ of type~I is immediate
from equation~\eqref{eq:xysize} and Proposition~\ref{prop:sphcomp}(a).
\end{proof}

%
%

\section{Basic Lemmas}
\label{sec:somelemmas}
\begin{lemma}\label{lem:SL0new}
Let $f\in\Cv(z)$. Suppose there exists $\delta >0$ such that
\[ \big(f^n\big)^{\nat}(\zeta) \geq \delta \quad \text{for all } \zeta\in\JBerkf
\text{ and } n\in\NN.  \]
Then there exist $\delta' >0$ and $h \in \PGL(2, \Cv)$ such that the map $g := h \circ f \circ h^{-1}$ satisfies:
\begin{itemize}
\item $|g(\zeta)|> 1$ for all $\zeta\in\PBerk$ with $|\zeta|> 1$,
\item $\dsps \JBerkg \subseteq\DbarBerk(0,1)$, and
\item $\dsps \big( g^n\big)^{\nat}(\zeta) \geq \delta'$
for all $\dsps \zeta\in\JBerkg$ and $n\in\NN$.
\end{itemize}
\end{lemma}

\begin{proof}
By \cite[Proposition~4.2]{BenBook}, there is a type~I point $\alpha\in\PCv$ that is a nonrepelling fixed point of $f$.
Let $h_1\in\PGL(2, \calO)$ be a M\"obius transformation satisfying $h_1(\alpha) = \infty$.
By \cite[Proposition~4.3(c)]{BenBook}, there is some $R>0$ so that the map $g_1:=h_1\circ f\circ h_1^{-1}$
satisfies $|g_1(x)|>R$ for all $x\in\PCv$ with $|x|>R$; and by \cite[Theorem~4.18]{BenBook}, we have $|g_1(x)|\geq |x|$
for all such $x$.

Choose $b\in\Cv^{\times}$ with $|b|\geq R$, and define $h_2\in\PGL(2,\Cv)$ by $h_2(z):= z/b$.
Define $h:=h_2\circ h_1$ and $g := h \circ f \circ h^{-1}$.
Then $|g(x)|\geq |x|$ for all $x\in\Cv$ with $|x|>1$, implying the first conclusion by continuity.
Moreover, the Fatou set $\FBerkg=h(\FBerkf)$ of $g$ must contain $\PBerk\smallsetminus\DbarBerk(0,1)$,
and hence the Julia set $\JBerkg=h(\JBerkf)$ is contained in $\DbarBerk(0,1)$, yielding the second conclusion.

It is easy to check that
\[ h_2^{\nat}(\zeta), \big( h_2^{-1} \big)^{\nat}(\zeta) \geq \min\big\{ |b|, |b|^{-1} \big\}
\quad \text{for all }\zeta\in\PBerk. \]
Therefore, by equation~\eqref{eq:natinvar} and Proposition~\ref{prop:chainrule},
we also have
\[ h^{\nat}(\zeta), \big( h^{-1} \big)^{\nat}(\zeta) \geq \min\big\{ |b|, |b|^{-1} \big\}
\quad \text{for all }\zeta\in\PBerk. \]
Define $\delta':=\delta\min\{|b|^2,|b|^{-2}\} >0$.
For any $\zeta\in\JBerkg$, again by Proposition~\ref{prop:chainrule}, we have
\begin{align*}
\big( g^n \big)^{\nat}(\zeta) &= h^{\nat}\Big( f^n\big(h^{-1}(\zeta)\big) \Big)
\cdot \big(f^n\big)^{\nat}\big( h^{-1}(\zeta) \big) \cdot \big(h^{-1}\big)^{\nat}(\zeta)
\\
& \geq \min\big\{ |b|, |b|^{-1} \big\} \cdot \big( f^n\big)^{\nat}\big( h^{-1}(\zeta) \big)
\cdot \min\big\{ |b|, |b|^{-1} \big\}
\geq \delta'. \qedhere
\end{align*}
\end{proof}

Define the real number
\[ \kappa := \begin{cases}
|p|^{1/(p-1)} & \text{ if } p > 0, \\
1  &\text{ if } p = 0, 
\end{cases} \]
where $p$ is the residue characteristic of $\Cv$.
Note that $0 <\kappa\leq 1$, since $\Cv$ itself has characteristic zero.

If a convergent power series $F \in \Cv[[z - a]]$ on a disk $D(a, r)$ has no critical points,
it is still possible that $F$ may not be injective on $D(a, r)$.
However, the next result shows that $F$ is injective on the smaller disk $D(a, \kappa r)$, 
scaling distances by a factor of $|F'(a)|$.

\begin{lemma}
\label{lem:kappa}
Fix $a\in\Cv$ and $r>0$. Let
\[ F(z) = \sum_{i=0}^{\infty} c_i (z-a)^i \in \Cv[[z-a]] \]
be a power series converging on $D(a,r)$.
If $F$ has no critical points in $D(a,r)$, then $F$
maps $D(a,\kappa r)$ bijectively onto $D(c_0, |c_1| \kappa r)$.
\end{lemma}

\begin{proof}
Because $F$ has no critical points in $D(a, r)$, the power series $F'$
has Weierstrass degree zero on this disk, and hence
\[ |n c_n| r^{n-1} \leq |c_1| \quad \text{for all }n \in \NN. \]
In addition, by definition of $\kappa$, we have $\kappa^{n-1} \leq |n|$, and hence
\[ |c_n| (\kappa r)^n \leq |n c_n| \kappa r^n \leq |c_1| (\kappa r) \quad \text{for all }n \in \NN. \]
Therefore, $F-c_0$ has Weierstrass degree $1$ and hence is injective on $D(a, \kappa r)$.
By \cite[Theorem~3.15]{BenBook}, $F$ maps $D(a,\kappa r)$ bijectively onto $D(c_0,|c_1| \kappa r)$.
\end{proof}

\begin{lemma}
\label{lem:SL1new}
Let $f\in\Cv(z)$ be a nonconstant rational map. Suppose that
all poles and (type~I) critical points of $f$ lie in the Fatou set $\FBerkf$.
Then there exists $\epsilon > 0$ such that for any point $a\in\Cv$ for which
$\DBerk(a,\epsilon)\cap\JBerkf\neq\varnothing$, we have that
$f$ maps $D(a,r)$ bijectively onto $D(f(a), |f'(a)|r)$
for any radius $r$ with $0<r\leq\epsilon$.
\end{lemma}

\begin{proof}
Denote by $\CP(f)$ the set of poles and (type~I) critical points of $f$.
Since each $c\in\CP(f)$ lies in $\FBerkf$, 
there is an associated radius $\delta_c>0$ such that
$\DBerk(c,\delta_c)\subseteq\FBerkf$. Because $\CP(f)$ is finite, we may define
\[ \epsilon_0:= \min\{ \delta_c \, | \, c\in\CP(f) \} >0
\quad\text{and}\quad \epsilon:=\kappa \epsilon_0 > 0  .\]

For any $a\in\Cv$ for which $\DBerk(a,\epsilon)$ intersects $\JBerkf$,
the larger disk $\DBerk(a,\epsilon_0)$ also intersects $\JBerkf$,
and hence cannot contain any points of $\CP(f)$. After all,
if $c\in \CP(f)$ lies in $\DBerk(a,\epsilon_0)$, then
$\DBerk(a,\epsilon_0) = \DBerk(c,\epsilon_0)$ is contained in $\FBerkf$,
a contradiction.

For such $a\in\Cv$, since $f$ has no poles in $D(a,\epsilon_0)$,
we may write $f$ as a power series
\[ f(z) = \sum_{i = 0}^{\infty} c_i (z - a)^i \in \Cv[[z-a]]  \]
converging on $D(a,\epsilon_0)$, with $c_0=f(a)$ and $c_1=f'(a)$.
By Lemma~\ref{lem:kappa}, for $0<r\leq\epsilon$, we have that
$f$ maps $D(a,r)$ bijectively onto $D(f(a), |f'(a)| r)$.
\end{proof}

%
%
%
%

\section{Technical Lemmas}
\label{sec:proofs}

To prepare for the proof itself, we need to set some notation
and present several technical lemmas.
Throughout this section, we assume $f\in\Cv(z)$ is as in Theorem~\ref{MAIN}.
By Lemma~\ref{lem:SL0new}, we may assume that $\JBerkf \subseteq \DbarBerk(0,1)$,
with $|f(x)| > 1$ for $|x| > 1$,
and such that $(f^n)^{\nat}(\zeta) \geq \delta$ for all $\zeta\in\JBerkf$ and $n\in\NN$.
Choose $\epsilon > 0$ as in Lemma~\ref{lem:SL1new}; thus, 
$f$ is injective on $D(a, \epsilon)$ for any $a\in\Cv$ for which
$\DBerk(a,\epsilon)\cap\JBerkf\neq\varnothing$.
Without loss, assume that $\delta,\epsilon < 1$.

For each $\zeta\in\JBerkf$, define the real quantities
\[ \sigma (\zeta) := \inf \big\{ (f^n)^{\nat}(\zeta) \, \big| \, n\in\NN_0 \big\},
\qquad
\nu(\zeta) := \frac{\delta^2 \epsilon}{\sigma (\zeta)}, \]
and for each $n\in\NN_0$,
\[ \mu_n( \zeta ) := \frac{\nu\big(f^n(\zeta)\big)}{ (f^n)^{\nat}(\zeta) \nu( \zeta ) }
= \frac{\sigma(\zeta)}{(f^n)^{\nat}(\zeta) \sigma\big(f^n(\zeta)\big)}. \]
The function $\sigma : \JBerkf \to \RR$ will serve as a local scaling factor
with respect to which $f$ will be everywhere expanding on $\JBerkf$
(see Lemma~\ref{lem:KL1new}.(b) below).

We also partition $\JBerkf$ into two pieces:
\begin{align*}
\JBerkf^+ & := \{\zeta\in\JBerkf \, | \, \sphdiam(\zeta) \geq \nu(\zeta) \}, \quad\text{and}
\\
\JBerkf^0 & := \{\zeta\in\JBerkf \, | \, \sphdiam(\zeta) < \nu(\zeta) \}.
\end{align*}
Moreover, for each $n\in\NN$, define $\JBerkf^n:= f^{-n}(\JBerkf^0)$.

Finally, we cover $\JBerkf^0$ with open disks, by setting
\[ \bOmega:=\bigcup_{\zeta\in\JBerkf^0} \DBerk\big(\zeta,\nu(\zeta)\big) \]
We will multiply the radii $\nu(\zeta)$ by the contracting factors $\mu_n(\zeta)$ to produce even smaller
neighborhoods of $\JBerkf^n$.


\begin{lemma}
\label{lem:KL1new}
For any $\zeta\in\JBerkf$, the following statements hold.
\begin{enumerate}
\item $\delta \leq \sigma(\zeta) \leq 1$ and
$\delta^2\epsilon \leq\nu(\zeta)\leq \delta\epsilon$.

\item $\dsps f^{\nat}(\zeta) \sigma \big( f(\zeta) \big) \geq \sigma (\zeta)$.
\end{enumerate}
\end{lemma}

\begin{proof}
\textbf{(a)}. For any $\zeta\in\JBerkf$, choosing $n=0$ in the definition of $\sigma(\zeta)$
yields the upper bound $\sigma(\zeta)\leq 1$. The lower bound follows from the hypothesis
that $(f^n)^{\nat}(\zeta)\geq \delta$ for all $n\in\NN_0$.
The bounds on $\nu$ follow immediately.

\medskip

\textbf{(b)}.
For any $n\in\NN_0$, we have
\[ f^{\nat}(\zeta) \cdot (f^n)^{\nat}\big(f(\zeta)\big)
= \big(f^{n+1}\big)^{\nat}(\zeta) , \]
by Proposition~\ref{prop:chainrule}.
Taking the infimum over all $n\in\NN_0$, we have
\begin{align*}
f^{\nat}(\zeta) \sigma \big( f(\zeta) \big) & = \inf\Big\{ \big(f^n\big)^{\nat}(\zeta) \, \Big| \, n\in\NN \Big\}
\\
& \geq \inf\Big\{ \big(f^n\big)^{\nat}(\zeta) \, \Big| \, n\in\NN_0 \Big\} = \sigma(\zeta).
\qedhere
\end{align*}
\end{proof}

\begin{lemma}
\label{lem:KL3new}
For each $\zeta \in \JBerkf$, we have
\begin{enumerate}
\item  $\dsps \mu_n( \zeta ) = \prod_{i = 0}^{n - 1} \mu_{1}\big(f^i(\zeta)\big)$ for all $n\in\NN$.
\item $1=\mu_0(\zeta) \geq \mu_1(\zeta) \geq \mu_2(\zeta)\geq \cdots$ .
\item If $\zeta\in\JBerkf^+$, then $\dsps f(\zeta)\in\JBerkf^+$.
\end{enumerate}
\end{lemma}

\begin{proof}
\textbf{(a)}.
Given $\zeta\in\JBerkf$ and $n\in\NN$, Proposition~\ref{prop:chainrule} yields
\begin{align*}
\prod_{i=0}^{n-1} \mu_1\big( f^i(\zeta) \big) &=
\prod_{i=0}^{n-1} \frac{ \sigma\big(f^i(\zeta) \big)}
{ f^{\nat}\big( f^i(z) \big) \sigma\big(f^{i+1}(z)\big) }
= \bigg(\prod_{i=0}^{n-1} \frac{ \sigma\big(f^i(\zeta) \big)}{\sigma\big(f^{i+1}(\zeta)\big) }\bigg) \cdot
\bigg(\prod_{i=0}^{n-1} \frac{1}{f^{\nat}\big( f^i(z) \big) } \bigg)
\\
& =  \frac{\sigma(\zeta)}{\sigma\big(f^n(\zeta)\big)} \cdot \frac{1}{(f^n)^{\nat}(\zeta)}= \mu_n(\zeta).
\end{align*}

\textbf{(b)}.
For any $\zeta\in\JBerkf$, clearly $\mu_0(\zeta)=1$.
Observe that
\[ \mu_1\big( f^i(\zeta) \big) \leq 1
\quad \text{for any } i\in\NN_0, \]
by Lemma~\ref{lem:KL1new}(b) applied to $f^i(\zeta)$. Thus, part~(a) of the current lemma
immediately implies part~(b).

\medskip

\textbf{(c)}.
For $\zeta\in\JBerkf^+$, we have
\begin{align*}
\sphdiam\big(f(\zeta)\big) &= f^{\nat}(\zeta)\sphdiam(\zeta) \geq f^{\nat}(\zeta) \nu(\zeta)
\\
& = \frac{f^{\nat}(\zeta)}{\sigma(\zeta)} \cdot \delta^2 \epsilon
\geq \frac{\delta^2\epsilon}{\sigma\big(f(\zeta)\big)} = \nu\big(f(\zeta)\big),
\end{align*}
where the first equality is by definition of $f^{\nat}$, the second and third equalities
are by definition of $\nu$, the first inequality is because $\zeta\in\JBerkf^+$,
and the second inequality is by Lemma~\ref{lem:KL1new}(b).
\end{proof}

It is immediate from Lemma~\ref{lem:KL3new}(c) that $\JBerkf^0\supseteq\JBerkf^1\supseteq\JBerkf^2\supseteq\cdots$.

\begin{lemma}\label{lem:KL4new}
For any $n \in \mathbb{N}_0$,
\[ f^{-n}(\bOmega) = \bigcup_{\zeta \in \JBerkf^n } \DBerk\big(\zeta, \mu_{n}(\zeta) \nu(\zeta) \big). \]
Moreover, we have
$\bOmega\supseteq f^{-1}(\bOmega) \supseteq f^{-2}(\bOmega)
\supseteq \cdots$.
\end{lemma}
%
%


\begin{proof}
We prove the equality by induction on $n$. It is trivial for $n=0$.
Assume it holds for some $n = m \in \NN_0$; we will prove it for $m+1$.

For the forward inclusion,
given $\xi \in f^{-(m + 1)}(\bOmega)$, there exists $\zeta\in\JBerkf^m$
such that $f(\xi) \in \DBerk(\zeta, \mu_m(\zeta)\nu(\zeta))$.
Write
\[ f^{-1}(\zeta) = \{ \theta_1,\ldots, \theta_d \} \subseteq \JBerkf ^{m+1}. \]
For each $i=1,\ldots, d$, Lemma~\ref{lem:KL1new}(a) yields
\[ \mu_{m+1}(\theta_i) \nu(\theta_i) = \frac{\nu\big(f^{m+1}(\theta_i)\big)}{ (f^{m+1})^{\nat}(\theta_i)}
\leq \frac{\delta\epsilon}{\delta} = \epsilon .\]
Therefore, by Lemma~\ref{lem:SL1new},
$f$ is injective on each disk $\DBerk(\theta_i, \mu_{m+1}(\theta_i) \nu(\theta_i))$,
scaling distances by a factor of $f^{\nat}(\theta_i)$.
Hence, the points $\theta_1,\ldots,\theta_d$ are indeed distinct, and
\begin{align*}
f\Big( \DBerk \big( \theta_i, \mu_{m+1}(\theta_i) \nu(\theta_i) \big) \Big)
&= \DBerk\big( f(\theta_i), f^{\nat}(\theta_i) \mu_{m+1}(\theta_i) \nu(\theta_i) \big) \\
&= \DBerk\Big( f(\theta_i), \mu_{m}\big(f(\theta_i)\big) \nu\big(f(\theta_i)\big) \Big) \\
&= \DBerk\big( \zeta, \mu_m(\zeta) \nu(\zeta) \big) .
\end{align*}
Since $\deg f=d$, we have accounted for all preimages of $\DBerk(\zeta,\mu_m(\zeta)\nu(\zeta))$.
Thus, there is some $j\in\{1,\ldots, d\}$ such that
\[ \xi \in \DBerk\big(\theta_j, \mu_{m+1}(\theta_j) \nu(\theta_j)\big), \]
completing our proof of the forward inclusion.

Conversely, given $\zeta\in\JBerkf^{m+1}$ and $\xi\in\DBerk(\zeta,\mu_{m+1}(\zeta) \nu(\zeta) )$, we have
\begin{align*}
f(\xi)
& \in f\Big( \DBerk\big( \zeta , \mu_{m+1}(\zeta) \nu(\zeta) \big) \Big)
= \DBerk\big( f(\zeta) , f^{\nat}(\zeta) \mu_{m+1}(\zeta) \nu(\zeta) \big) \\
&= \DBerk\Big( f(\zeta) , \mu_m\big( f(\zeta) \big) \nu\big( f(\zeta) \big) \Big)
\subseteq f^{-m}(\bOmega),
\end{align*}
verifying the reverse inclusion.

Finally, for any $n\in\NN_0$,
we have $\mu_{n+1}(\zeta)\leq \mu_n(\zeta)$
for all $\zeta\in\JBerk^{n+1}$, by Lemma~\ref{lem:KL3new}(b).
Since $\JBerk^{n+1}\subseteq\JBerk^n$,
it follows immediately that $f^{-n-1}(\bOmega)\subseteq f^{-n}(\bOmega)$.
\end{proof}

\begin{lemma}
\label{lem:KL5}
We have
\[ \bigcap_{n\in\NN_0} f^{-n}(\bOmega) = \bigcap_{n\in\NN_0} \JBerkf^n = \JBerkf\cap\Cv .\]
Moreover, for any $\zeta\in\JBerkf\cap\Cv$, we have $\dsps \lim_{n \to \infty} \mu_n(\zeta) = 0$.
\end{lemma}

\begin{proof}
The inclusion $(\supseteq)$ in the first equality is immediate from
the definitions of $\bOmega$ and $\JBerkf^n$,
and the inclusion $(\supseteq)$ in the second equality is because $\sphdiam(f^n(\zeta))=0$
for every point $\zeta$ of type~I and every $n\in\NN_0$. Thus, to show these two equalities,
it suffices to show that the first set is contained in the third.

Given $\xi\in\bigcap_{n=0}^{\infty} f^{-n}(\bOmega)$, by Lemma~\ref{lem:KL4new}, there is a sequence
of points $\{\zeta_n\}_{n=0}^{\infty}$ such that for every $n\in\NN_0$, we have
\begin{equation}
\label{eq:zetandef}
\zeta_n\in \JBerkf^n \quad\text{and}\quad \xi \in \DBerk\big( \zeta_n, \mu_n(\zeta_n) \nu(\zeta_n) \big)
\end{equation}
Define
\[ t:= \inf \big\{ \mu_n(\zeta_n) \nu(\zeta_n) \, \big| \, n\in\NN_0 \big\} \geq 0. \]

We claim that $t=0$. If not, i.e., if $t>0$, then there is some $j\in\NN_0$ such that
$t > \delta^{1/2} \mu_j(\zeta_j) \nu(\zeta_j)$.
There must be some $m\in\NN$ such that $f^{m}(\DBerk(\zeta_j,\mu_j(\zeta_j)\nu(\zeta_j)))$
is not contained in an open disk of radius $\epsilon$,
or else $f^i(\DBerk(\zeta_j,\mu_j(\zeta_j)\nu(\zeta_j)))\subseteq \DbarBerk(0,1)$ for all $i\in\NN_0$,
contradicting the fact that $\zeta_j\in\JBerkf$. Let $m$ be the smallest such integer.
Since $f^{i}(\DBerk(\zeta_j,\mu_j(\zeta_j)\nu(\zeta_j)))$ is contained in $\DBerk(f^i(\zeta_j),\epsilon)$
for every $0\leq i<m$, repeated application of Lemma~\ref{lem:SL1new} shows that
$f^m$ maps $\DBerk(\zeta_j,\mu_j(\zeta_j)\nu(\zeta_j))$ bijectively onto 
a disk of radius greater than $\epsilon$.

Choose a point $\theta\in\PBerk$ as follows. If $\diam(\xi) \geq \delta^{1/2} t$, then choose
$\theta:=\xi$; otherwise, choose $\theta$ to be the unique boundary point of the disk
$\DBerk(\xi,\delta^{1/2} t)$. Then
\[ \sphdiam(\theta)=\diam(\theta)\geq \delta^{1/2} t > \delta \mu_j(\zeta_j) \nu(\zeta_j),\]
and in addition, for every $i\in\NN_0$, we have
$\theta\in\DBerk(\zeta_i,\mu_i(\zeta_i)\nu(\zeta_i))$.

Because $\theta$ lies in the disk $\DBerk(\zeta_j,\mu_j(\zeta_j)\nu(\zeta_j))$,
Lemma~\ref{lem:injdiam} applied to $f^m$ implies that
\begin{equation}
\label{eq:thetadiam}
\sphdiam\big( f^m(\theta) \big) >  \frac{\epsilon}{\mu_j(\zeta_j)\nu(\zeta_j)} \cdot \diam(\theta)
> \delta\epsilon \geq \nu\big( f^m(\zeta_m) \big),
\end{equation}
where the last inequality is by Lemma~\ref{lem:KL1new}(a).
However, since $\theta$ also lies in the disk $\DBerk(\zeta_m, \mu_m(\zeta_m)\nu(\zeta_m))$,
we have
\[ f^m(\theta) \in f^m\Big(  \DBerk\big(\zeta_m, \mu_m(\zeta_m)\nu(\zeta_m) \big) \Big)
= \DBerk\Big( f^m(\zeta_m) , \nu\big( f^m(\zeta_m) \big) \Big). \]
Therefore,
$\sphdiam( f^m(\theta)) = \diam(f^m(\theta)) \leq \nu( f^m(\zeta_m))$,
contradicting inequality~\eqref{eq:thetadiam}.
Our claim follows; we must have $t=0$.

The point $\xi$ is therefore contained in disks $\DBerk( \zeta_n, \mu_n(\zeta_n) \nu(\zeta_n) )$
of arbitrarily small positive radius. Thus, $\diam(\xi)=0$, implying that $\xi\in\Cv$.
The points $\zeta_n\in\JBerkf$ accumulate at $\xi$, and hence $\xi\in\JBerkf\cap\Cv$, as desired.

Finally, given $\xi\in\JBerkf\cap\Cv$, we may choose the sequence $\{\zeta_n\}_{n=0}^{\infty}$
of~\eqref{eq:zetandef} to be the constant sequence $\zeta_n:=\xi$.
Since the sequence $\{\mu_n(\xi)\}_{n=0}^{\infty}$ is decreasing by Lemma~\ref{lem:KL3new}(b),
the above claim immediately yields $\mu_n(\xi)\to 0$.
\end{proof}

\begin{remark} 
\label{rem:noperiodic}
One consequence of Lemma~\ref{lem:KL5} is that every
periodic Julia point in $\HBerk$ belongs to $\JBerkf^+$.
Indeed, if $\zeta \in \JBerkf^0\cap\Omega$ is periodic
of period $\ell\geq 1$, then Lemma~\ref{lem:KL5} gives
\[ \zeta \in \bigcap_{n \in \NN_0}f^{-n\ell}(\Omega)
= \bigcap_{n \in \NN_0}f^{-n}(\Omega) = \JBerkf \cap \Cv, \]
where we have also applied Lemma~\ref{lem:KL4new}.
In particular, there is a uniform lower bound for
the spherical diameter the set of
Julia periodic points that are not of type~I.

On the other hand, such a lower bound does not hold
in general without the bounded contraction hypothesis
of Theorem~\ref{MAIN}.
For instance, the map of Example~10.20 of \cite{BenBook} has
an infinite sequence of attracting periodic points $a_n\in\Cv$
accumulating at a type~I Julia point $b$. Each $a_n$ must lie
in a different Fatou disk with a single type~II
repelling periodic point $\zeta_n$ as its boundary.
We have $\zeta_n\to b$, and hence the diameters of
the type~II Julia periodic points $\zeta_n$ must approach zero.
\end{remark} 

\begin{lemma}
\label{lem:nbrhd}
For any $\gamma>0$,
there is an open subset $W$ of $\Rat_d(\Cv)$ containing $f$ such that
for any $g\in W$, we have
\begin{enumerate}
\item $|g(x)|>1$ for any $x\in\Cv$ with $|x|>1$, and
\item $\dsps | g(x) - f(x) | < \gamma$ for all $x\in f^{-1}(\Dbar(0,1))$.
\end{enumerate}
\end{lemma}

\begin{proof}
Write $f=F/G$ for relatively prime polynomials $F,G\in \Cv[z]$, with
\[ F(z) = a_d z^d + \cdots + a_0 \quad\text{and}\quad G(z) = b_d z^d + \cdots + b_0 ,\]
and write an arbitrary $g\in\Rat_d(\Cv)$ as $\tilde{F}/\tilde{G}$, with $\tilde{F}, \tilde{G}\in\Cv[z]$ given by
\begin{equation}
\label{eq:gcoeffs}
\tilde{F}(z) = A_d z^d + \cdots + A_0 \quad\text{and}\quad \tilde{G}(z) = B_d z^d + \cdots + B_0 .
\end{equation}
As we assumed at the start of this section,
we have $|f(x)|>1$ for any $x\in\Cv$ with $|x|>1$.
Therefore, $a_d \neq 0$, and $|a_i|,|b_i|\leq |a_d|$ for all $i=0,\ldots, d$.
Let $W_1$ be the subset of $\Rat_d(\Cv)$ defined by the (open) conditions
\[ |A_i-a_i|<|a_d| \quad\text{and}\quad |B_i-b_i|<|a_d| \quad\text{for all } i=0,\ldots, d .\]
Then any $g\in W_1$ has $|g(x)|>1$ for any $x\in\Cv$ with $|x|>1$.

Let $y_1,\ldots,y_{\ell}$ denote the distinct poles of $f$ in $\Dbar(0,1)$,
and choose a radius $0<r<1$ so that
$f(D(y_i,r))\subseteq\PCv\smallsetminus\Dbar(0,1)$ for each $i$.
Then
\[ |G(x)| \geq C \quad \text{for all }
x\in \Dbar(0,1) \smallsetminus \big( D(y_1, r) \cup \cdots\cup D(y_\ell, r) \big), \]
where
\[ C:= \min\{\|G\|_{\zeta(y_1,r)},\ldots, \|G\|_{\zeta(y_\ell,r)} \} >0
\quad\text{if } \ell\geq 1,\]
or $C:=|G(0)|=\|G\|_{\zeta(0,1)}$ if $\ell=0$.


With notation as in equations~\eqref{eq:gcoeffs}, define $W_2$ to be
the open neighborhood of $f$ in $\Rat_d(\Cv)$ given by the conditions
\[ |A_i - a_i| < \frac{C^2 \gamma}{|a_d|} \quad\text{and}\quad  |B_i - b_i|
< \min \bigg\{ \frac{C^2\gamma}{|a_d|}, C \bigg\} \quad \text{for each } i. \]
Then any $g= \tilde{F}/\tilde{G} \in W_2$ satisfies
\[ \big|\tilde{F}(x)-F(x)\big| < \frac{C^2 \gamma}{|a_d|} \quad \text{and}\quad
\big|\tilde{G}(x) -G(x)\big| <  \min \bigg\{ \frac{C^2\gamma}{|a_d|}, C \bigg\}
\quad \text{for all }x\in\Dbar(0,1). \]
Therefore, for any $g=\tilde{F}/\tilde{G} \in W_2$ and $x\in\Cv$ such that $|f(x)|\leq 1$, we have
\[ \big|\tilde{G}(x) - G(x)\big| < C \leq |G(x)|,
\quad\text{and hence}\quad \big|\tilde{G}(x)\big|=|G(x)| \geq C. \]
Thus,
\begin{align*}
\big|g(x) - f(x)\big| &=
\frac{\big| G(x) \big(\tilde{F}(x)-F(x)\big) - F(x) \big(\tilde{G}(x)-G(x)\big) \big|}{|G(x)|\cdot |\tilde{G}(x)|}
\\
&\leq \frac{1}{C^2} \max\big\{ |a_d| |\tilde{F}(x) - F(x)|, |a_d| |\tilde{G}(x)-G(x)| \big\}
\\
& < \frac{1}{C^2} \max\big\{ C^2\gamma, C^2\gamma \big\} = \gamma
\end{align*}
Finally, defining $W:=W_1\cap W_2$, we are done.
\end{proof}

\begin{lemma}
\label{lem:inj}
Let $W\subseteq\Rat_d(\Cv)$ be the open neighborhood of $f$ from
Lemma~\ref{lem:nbrhd} for some $\gamma$ with $0<\gamma <\delta^2\epsilon$.
Then for any $g\in W$ and any $\zeta\in\JBerkf^1$, we have $f^{\nat}(\zeta) = g^{\nat}(\zeta)$, and
\[ g \text{ maps } D\big(\zeta, \mu_1(\zeta)\nu(\zeta) \big)
\text{ bijectively onto } D\Big( f(\zeta), \nu\big( f(\zeta) \big) \Big) .\]
Moreover, $g^{-1}(\bOmega) = f^{-1}(\bOmega)$.
\end{lemma}

\begin{proof}
Given $g$ and $\zeta$ as specified, let $r:=\mu_1(\zeta)\nu(\zeta)$,
so that
\[ \diam(\zeta) = \sphdiam(\zeta) < r \leq \delta \epsilon, \]
by Lemmas~\ref{lem:KL1new}(a) and~\ref{lem:KL3new}(b).
Choose $x\in\Cv$ with $\|z-x\|_\zeta < r$, so that $\zeta\in\DBerk(x,r)\subseteq\bOmega$.
Recall that $f$ has no poles in $\bOmega$,
and hence neither does $g$, by the defining property of $W$ in Lemma~\ref{lem:nbrhd}(b).
Thus, we may expand both $f$ and $g$ as power series
\[ f(z) = \sum_{i=0}^{\infty} a_i (z-x)^i
\quad\text{and}\quad g(z) = \sum_{i=0}^{\infty} b_i (z-x)^i \]
convergent on $D(x,r)$.
Because $r\leq\epsilon$, Lemma~\ref{lem:SL1new} implies that $f-a_0$ is injective on $D(x,r)$
and hence has Weierstrass degree $1$.
That is, $|a_i| r^i \leq |a_1|r$ for all $i\geq 1$. We will show the analogous statement for $g-b_0$.

By the defining property of $W$, the power series
\[ g(z) - f(z) = \sum_{i=0}^{\infty} (b_i-a_i) (z-x)^i \]
satisfies $|g(y)-f(y)|<\gamma$ for $y\in D(x,r)$, and hence
\begin{equation}
\label{eq:biai}
\big|b_i - a_i\big| r^i \leq \gamma
\quad\text{for all } i\in\NN_0.
\end{equation}
On the other hand, we have $|a_1|=|f'(x)|=f^{\nat}(\zeta)$, and therefore
\[ |a_1|r = f^{\nat}(\zeta) \mu_1(\zeta) \nu(\zeta) = \nu\big(f(\zeta)\big) \geq \delta^2\epsilon > \gamma . \]
Combined with~\eqref{eq:biai} for $i=1$, it follows that $|b_1-a_1| < |a_1|$, and hence
\begin{equation}
\label{eq:gfprime}
g^{\nat}(\zeta)=|g'(x)| = |b_1|=|a_1| = f^{\nat}(\zeta) .
\end{equation}
Furthermore, applying~\eqref{eq:biai} for $i\geq 1$, we have
\[ |b_i| r^i  \leq \max \big\{ \big| b_i - a_i \big| r^i, |a_i| r^i \big\} \\
\leq \max \big\{ \gamma , |a_1|r \big\} = |b_1|r \quad \text{for all } i\geq 1. \]
That is, $g-b_0$ has Weierstrass degree 1 on $D(x,r)$.

Thus, $g$ maps $D(\zeta,r)=D(x,r)$ bijectively onto $D(g(x), |g'(x)| r)$.
However, since $|b_1|=|a_1|$ by equation~\eqref{eq:gfprime}, we have
\[ |g'(x)| r = |b_1| r = |a_1|r = \nu\big( f(\zeta)\big). \]
Hence,
\[ \big|g(y)-f(y)\big|<\gamma\leq \delta^2\epsilon\leq \nu \big(f(\zeta)\big) = \big| g'(x)\big| r
\quad\text{for all } y\in D(x,r). \]
Therefore, the image of $D(\zeta,r)$ under $g$ is
\[ D\big( g(x), |g'(x)| r \big) = D\Big( f(x), \nu\big( f(\zeta) \big) \Big) = D\Big( f(\zeta), \nu\big( f(\zeta) \big) \Big) .\]

Lastly, we must show that $g^{-1}(\bOmega) = f^{-1}(\bOmega)$.
For any $\zeta\in\JBerkf^0$,
let $\theta_1,\ldots,\theta_d\in\JBerkf^1$ be the $d$ preimages of $\zeta$ under $f$,
which we know to be distinct as in the proof of Lemma~\ref{lem:KL4new}.
By the first part of the current lemma,
we also know that $g$ maps each disk $D(\theta_i, \mu_1(\theta_i)\nu(\theta_i))$ bijectively
onto $D(\zeta,\nu(\zeta))$, accounting for all $d$ preimages of $D(\zeta,\nu(\zeta))$
under $g$. Thus,
\begin{align*}
g^{-1}\Big( \DBerk\big(\zeta,\nu(\zeta)\big)\Big)
&= \DBerk\big(\theta_1, \mu_1(\theta_1) \nu(\theta_1)\big)
\cup \cdots \cup \DBerk\big(\theta_d, \mu_1(\theta_d) \nu(\theta_d)\big)
\\
&= f^{-1}\Big( \DBerk\big(\zeta,\nu(\zeta)\big)\Big) .
\end{align*}
Taking the union across all $\zeta\in\JBerkf^0$, we have $g^{-1}(\bOmega) = f^{-1}(\bOmega)$.
\end{proof}

\section{Proof of Theorem~\ref{MAIN}}
\label{sec:main}

With notation as in Section~\ref{sec:proofs}, we are now prepared to prove our main result, as follows.
In Step~1, we define a sequence $\{ h_n \}_{n=0}^{\infty}$ of maps from subsets of $\bOmega$ to $\bOmega$,
and we investigate properties of this sequence in Step~2.
Then, in Step~3, we glue the maps $h_n$ to produce the desired map $h:\PBerk\to\PBerk$ 
that is a conjugacy on $f^{-1}(\bOmega)$.
In Steps~4 and~5, we show that $h$ is a homeomorphism on $\PBerk$, and that the conjugacy extends to $f^{-1}(\bOmega)\cup\JBerkf$.
Finally, in Step~6, we show that $h$ varies continuously with $g$. 

\begin{proof}[Proof of Theorem~\ref{MAIN}]
\textbf{Step~1}.
Fix a real number $0<t<1$, and let $W=W_t(f)$ be the neighborhood $W$ of $f$ given by
Lemma~\ref{lem:nbrhd} for $\gamma=t \delta^2\epsilon$. 
For the rest of this proof, consider an arbitrary map $g\in W$.

By Lemma~\ref{lem:inj}, for each $\zeta\in\JBerkf^1$, the map $g$ is injective
on $D(\zeta,\mu_1(\zeta)\nu(\zeta))$, with image
$D(f(\zeta), \nu(f(\zeta)))$. Thus, there exists a map
\[ G_{\zeta} : D\Big(f(\zeta), \nu\big(f(\zeta)\big)\Big)
\to D\big(\zeta,\mu_1(\zeta)\nu(\zeta)\big) \]
which is an inverse to $g$ given by a power series convergent on $D(f(\zeta), \nu(f(\zeta)))$.
Note that if $\xi\in\JBerkf^1$ lies in the same disk $D(\zeta,\mu_1(\zeta)\nu(\zeta))$, then
the power series $G_\zeta$ and $G_\xi$ agree, since $g$ is injective on both 
$D(\zeta,\mu_1(\zeta)\nu(\zeta))$ and $D(\xi,\mu_1(\xi)\nu(\xi))$.
As usual, the power series defining $G_{\zeta}$ extends via continuity to
\begin{equation}
\label{eq:Gzdef}
G_{\zeta} : \DBerk\Big(f(\zeta), \nu\big(f(\zeta)\big)\Big)
\to \DBerk\big(\zeta,\mu_1(\zeta)\nu(\zeta)\big)
\end{equation}

We now define a sequence $\{h_n\}_{n=0}^\infty$ of functions, with $h_0:\PBerk\to\PBerk$ and
\[ h_n : f^{-n}(\bOmega) \to \bOmega \quad\text{for } n\geq 1\]
by the following inductive method.
Let $h_0: \PBerk\to\PBerk$ by $h_0(\zeta):=\zeta$.
For each $n \in \mathbb{N}$, having already defined $h_{n-1}$, we define $h_n$ as follows.
For each $\zeta\in\JBerkf^n$,
define $h_n$ on $\DBerk(\zeta,\mu_n(\zeta)\nu(\zeta))$ by
\[ h_{n} := G_{\zeta} \circ h_{n - 1} \circ f, \]
where $G_{\zeta}$ is the local inverse of $g$ defined in~\eqref{eq:Gzdef}.

\medskip

\textbf{Step~2}.
We will now show that for each $n\in\NN_0$,
\begin{itemize}
\item $h_n$ is a well-defined function mapping
$\dsps f^{-n}(\bOmega)$ bijectively onto $g^{-n}(\bOmega)$,
given by a convergent power series on each disk 
$\DBerk(\zeta, \mu_n(\zeta)\nu(\zeta))$ for $\zeta\in\JBerkf^n$,
\item $h_n$ is an isometry on $f^{-n}(\bOmega)\cap\Cv$, and
\item for every $\zeta\in\JBerkf^n$ and $x\in D(\zeta, \mu_n(\zeta)\nu(\zeta))$,
we have $\big| h_n(x) - x \big| < t \mu_1(\zeta)\nu(\zeta)$, with 
\begin{equation}
\label{eq:hncauchy}
\big| h_{n}(x) - h_{n-1}(x) \big| < t \mu_{n}(\zeta) \nu(\zeta) \quad\text{if } n\geq 1, 
\end{equation}
where $0<t<1$ is the constant we fixed at the start of Step~1. 
\end{itemize}
We proceed by induction. For $n=0$, all three properties hold trivially.

For $n\geq 1$, assume the three bullet points hold for $n-1$.
Then $h_n$ is well-defined because if $\xi$ lies in both
$\DBerk(\zeta,\mu_1(\zeta)\nu(\zeta))$ and $\DBerk(\zeta',\mu_1(\zeta')\nu(\zeta'))$,
then as we noted in Step~1, the power series $G_{\zeta}$ and $G_{\zeta'}$ agree.
That is, the value of $h_n(\xi)$ is independent of which point $\zeta$ is chosen
as the center of the disk.

For any $\zeta\in\JBerkf^n$, it is immediate from Proposition~\ref{prop:chainrule}
and the definition of $\mu_n$ that
\[ f^{\nat}(\zeta)\mu_n(\zeta)\nu(\zeta) =
\frac{\nu\big( f^n(\zeta)\big)}{(f^{n-1})^{\nat}\big( f(\zeta)\big)}
=\mu_{n-1}\big( f(\zeta) \big) \nu\big( f(\zeta) \big), \]
and hence, by Lemmas~\ref{lem:injdiam} and~\ref{lem:SL1new},
$f$ is a convergent power series on the disk $U_n:= D(\zeta, \mu_n(\zeta)\nu(\zeta))$, mapping
\[ D\big(\zeta, \mu_n(\zeta)\nu(\zeta)\big)
\quad\text{bijectively onto}\quad
D\Big(f(\zeta), \mu_{n-1}\big(f(\zeta)\big)\nu\big(f(\zeta)\big) \Big), \]
and multiplying all distances by a factor of $f^{\nat}(\zeta)$.
By our inductive assumptions, $h_{n-1}$ acts as a power series mapping
$D(f(\zeta), \mu_{n-1}(f(\zeta))\nu(f(\zeta)))$
isometrically onto
\[ V_{n-1}:= D\Big(h_{n-1}\big(f(\zeta)\big), \mu_{n-1}\big(f(\zeta)\big)\nu\big(f(\zeta)\big) \Big)
\subseteq D\Big(f(\zeta), \nu\big(f(\zeta)\big) \Big), \]
where the inclusion is because of the inductive assumption that
\[ \big|h_{n-1}\big(f(x)\big)-f(x)\big| < t \mu_1\big(f(\zeta)\big)\nu\big(f(\zeta)\big) < \nu\big( f(\zeta) \big) 
\quad\text{for all } x\in V_{n-1}. \]
Thus, $G_{\zeta}$ is defined as an injective power series on the disk
$V_{n-1}=h_{n-1}( f(U_n))$, multiplying all distances by $(g^{\nat}(\zeta))^{-1} = (f^{\nat}(\zeta))^{-1}$,
where this equality is by Lemma~\ref{lem:inj}.
Therefore, $h_n=G_{\zeta} \circ h_{n-1}\circ f$ is a power series on $U_n$, mapping
\begin{equation}
\label{eq:hnisom}
D\big( \zeta, \mu_n(\zeta) \nu(\zeta) \big)
\quad\text{isometrically onto}\quad
D\big( h_n(\zeta), \mu_n(\zeta) \nu(\zeta) \big) .
\end{equation}

We will prove $h_n$ is an isometry on all of $f^{-n}(\bOmega)\cap\Cv$ shortly, but first
we prove the third bullet point for our given $n\geq 1$.
Given $\zeta\in\JBerkf^n$ and $x\in D(\zeta,\mu_n(\zeta)\nu(\zeta))$,
we first claim that
\begin{equation}
\label{eq:hntemp}
\big| h_{n-1} \big( f(x) \big) - g \big( h_{n-1} (x) \big) \big| <
t \mu_{n-1}\big( f(\zeta) \big) \nu\big(f(\zeta)\big).
\end{equation}
Indeed, if $n=1$, we have $|f(x)-g(x)| < t \delta^2\epsilon \le t\nu(f(\zeta))$ 
by Lemmas~\ref{lem:KL1new}(a) and~\ref{lem:nbrhd},
yielding~\eqref{eq:hntemp}.
If $n\geq 2$, we have $g(h_{n-1}(x))=h_{n-2}(f(x))$, and by our inductive
assumption for $f(x)$, we also have
\[ \big| h_{n-1} \big( f(x) \big) -  h_{n-2} \big( f(x) \big) \big|
< t \mu_{n-1} \big( f(\zeta) \big) \nu\big( f(\zeta) \big),\] 
proving~\eqref{eq:hntemp}.
Moreover, $h_{n-1}(f(x))$ and $g(h_{n-1}(x))$ both lie in
$D(f(\zeta), \nu(f(\zeta)))$, and hence we may apply $G_{\zeta}$.
Recalling that $G_{\zeta}$ scales distances by $(f^{\nat}(\zeta))^{-1}$, we have
\begin{equation}
\label{eq:hnclose}
\big| h_n(x) - h_{n-1}(x) \big| < t (f^{\nat}(\zeta))^{-1}
\mu_{n-1} \big( f(\zeta) \big) \nu\big( f(\zeta) \big) = t \mu_n(\zeta) \nu(\zeta), 
\end{equation}
giving inequality~\eqref{eq:hncauchy}. The first part of the third bullet point
then follows from this bound together with the inductive assumption, because
\begin{align*}
\big| h_n(x) -x \big|
& \leq \max \big\{ \big|h_n(x) - h_{n-1}(x)\big| , \big| h_{n-1}(x) - x \big| \big\}
\\
& < \max \big\{ t \mu_n(\zeta) \nu(\zeta) , t \mu_1(\zeta)\nu(\zeta) \big\} = t \mu_1(\zeta)\nu(\zeta), 
\end{align*}
where the final equality is by Lemma~\ref{lem:KL3new}(b).

As for the second bullet point, that $h_n$ is an isometry on $f^{-n}(\bOmega)\cap\Cv$,
consider arbitrary $x,y\in f^{-n}(\bOmega)\cap\Cv$. Then there exist $\zeta,\xi\in\JBerkf^n$
such that $x\in D(\zeta,\mu_n(\zeta)\nu(\zeta))$ and $y\in D(\xi,\mu_n(\xi)\nu(\xi))$, by Lemma~\ref{lem:KL4new}.
Without loss, $\mu_n(\zeta)\nu(\zeta)\geq \mu_n(\xi)\nu(\xi)$.

If $|x-y|<\mu_n(\zeta)\nu(\zeta)$, then we have $|h_n(x)-h_n(y)|=|x-y|$ by~\eqref{eq:hnisom}.
Otherwise,
\[ \big| h_n(x) - h_{n-1}(x) \big|
< t \mu_n(\zeta)\nu(\zeta) < |x-y| = \big| h_{n-1}(x) - h_{n-1}(y) \big|, \] 
and similarly for $\big| h_n(y) - h_{n-1}(y) \big|$,
where the first inequality is by~\eqref{eq:hnclose}, and the equality is
by our inductive assumption. Thus,
\begin{align*}
\big| h_n(x) - h_n(y) \big| & = \Big|
\big( h_n(x) - h_{n-1}(x) \big) - \big( h_n(y) - h_{n-1}(y) \big) + h_{n-1}(x) - h_{n-1}(y)
\Big| \\
& = \big| h_{n-1}(x) - h_{n-1}(y) \big| = |x-y|,
\end{align*}
as desired.

It remains to show that $h_n$ maps $f^{-n}(\bOmega)$ bijectively onto $g^{-n}(\bOmega)$.
Because $h_n$ is given locally by power series, it suffices to show that $h_n$ maps
$f^{-n}(\bOmega)\cap\Cv$ bijectively onto $g^{-n}(\bOmega)\cap\Cv$.

Since $h_n$ is an isometry, we already know it is injective on $f^{-n}(\bOmega)\cap\Cv$.
In addition, for any $x\in f^{-n}(\bOmega)\cap\Cv$, we have $f(x)\in f^{-(n-1)}(\bOmega)$,
and therefore by our inductive assumption, we also have
\[ h_{n-1}\big( f(x)\big )\in g^{-(n-1)}(\bOmega). \]
Since each map $G_{\zeta}$ is a local inverse of $g$, it follows that
$h_n(x) \in g^{-n}(\bOmega)$.

Finally, given $y \in g^{-n}(\bOmega)\cap\Cv$, we have $g(y)\in g^{-(n-1)}(\bOmega)$,
and hence there is some $\tilde{x}\in f^{-(n-1)}(\bOmega)$ such that
$h_{n-1}(\tilde{x})=g(y)$, by our inductive assumption.
By Lemma~\ref{lem:KL4new}, there is some $\zeta\in\JBerkf^{n-1}$ such that
$\tilde{x}\in D(\zeta, \mu_{n-1}(\zeta) \nu(\zeta))$. 
Writing $f^{-1}(\zeta) = \{\theta_1,\ldots, \theta_d\}$, each disk
\[ D\big(\theta_i, \mu_1(\theta_i)\nu(\theta_i)\big)
\quad\text{ maps bijectively onto }\quad D\big(\zeta , \nu(\zeta)\big) \]
under both $f$ and $g$, by Lemmas~\ref{lem:SL1new} and~\ref{lem:inj}.
Moreover, because
\[ \big| g(y) - \tilde{x} \big| = \big| h_{n-1}(\tilde{x}) - \tilde{x} \big| < t \mu_1(\zeta)\nu(\zeta) < \mu_1(\zeta)\nu(\zeta)\] 
by our inductive assumption, we have
$g(y)\in D(\zeta, \mu_1(\zeta) \nu(\zeta))$.
Therefore, there is some $j\in\{1,\ldots, d\}$ such that
$y\in D(\theta_j, \nu(\theta_j))$, and there is some $x\in D(\theta_j, \nu(\theta_j))$
such that $f(x)=\tilde{x}$. Since $\tilde{x}\in f^{-(n-1)}(\bOmega)$,
we have $x\in f^{-n}(\bOmega)$.
Writing $\theta:=\theta_j$, we have $G_{\theta}(g(y))=y$, and hence $h_n(x)=y$.
Thus, $h_n$ does indeed map $f^{-n}(\bOmega)\cap\Cv$ bijectively onto
$g^{-n}(\bOmega)\cap\Cv$, completing our induction.

\medskip

\textbf{Step 3}.
For each $n\in\NN_0$, define $H_n:\PBerk\to\PBerk$ by the following inductive
procedure. Let $H_0=h_0$, and for $n\in\NN$ and $\zeta\in\PBerk$, let
\[ H_n(\zeta) := \begin{cases}
H_{n-1}(\zeta) & \text{ if } \zeta\in\PBerk\smallsetminus f^{-n}(\bOmega),
\\
h_n(\zeta) & \text{ if } \zeta\in f^{-n}(\bOmega).
\end{cases} \]
Define $h:\PBerk\to\PBerk$ by
\[ h(\zeta) := \lim_{n\to\infty} H_n(\zeta), \]
or equivalently
\[ h(\zeta) = \begin{cases}
\zeta & \text{ if } \zeta\in\PBerk\smallsetminus f^{-1}(\bOmega), \\
h_n(\zeta) & \text{ if } \zeta\in f^{-n}(\bOmega)\smallsetminus f^{-(n+1)}(\bOmega)
\text{ for } n\in\NN,
\\
\lim_{n\to\infty} h_n(\zeta) & \text{ if } \zeta\in  \bigcap_{n\in\NN} f^{-n}(\bOmega).
\end{cases} \]
For the third case, recall from Lemma~\ref{lem:KL5} that
$\bigcap_{n\in\NN_0} f^{-n}(\bOmega) = \JBerkf\cap\Cv$,
and that  $\lim_{n\to\infty} \mu_n(\zeta)=0$ for such $\zeta$.
Thus, by the third bullet point of Step~2, the sequence $\{h_n(\zeta)\}_{n=0}^{\infty}$
is Cauchy and hence converges to $h(\zeta)\in\bOmega\cap\Cv$.
Together with Lemmas~\ref{lem:KL4new} and~\ref{lem:inj}, as well as the first bullet point of Step~2,
it follows that $h$ is indeed a function from $\PBerk$ to itself.
Moreover, by the second bullet point of Step~2,
$h$ maps $f^{-n}(\bOmega)\cap\Cv$ bijectively onto $g^{-n}(\bOmega)\cap\Cv$ for each $n\in\NN_0$.

We claim that $h$ is an isometry on $\Cv$. To see this, given $x,y\in\Cv$,
we consider several cases. First, if $x,y\in\JBerkf$, then
\[\big|h(x) -h(y)\big| = \Big| \lim_{n\to\infty} h_n(x) - h_n(y) \Big|
= \lim_{n\to\infty} \big| h_n(x) - h_n(y) \big| = \lim_{n\to\infty} |x-y| = |x-y|, \]
where the third equality is because $h_n$ is an isometry on $f^{-n}(\bOmega)\cap\Cv$, by Step~2.
Second, if $x,y\in f^{-n}(\bOmega)\smallsetminus f^{-(n+1)}(\bOmega)$
for some $n\in\NN$, or if $x,y\in\PBerk\smallsetminus f^{-1}(\bOmega)$ with $n=0$, then
\[ \big|h(x) -h(y)\big| = \big| h_n(x) - h_n(y) \big| = |x-y|. \]
Finally, suppose there is some $n\in\NN_0$ such that 
\begin{equation}
\label{eq:diffparts}
x\in \begin{cases}
\PBerk\smallsetminus f^{-1}(\bOmega) & \text{ if } n=0,
\\
f^{-n}(\bOmega)\smallsetminus f^{-(n+1)}(\bOmega) & \text{ if } n\geq 1,
\end{cases}
\end{equation}
and $y\in f^{-(n+1)}(\bOmega)$.
Then for every $m>n$
for which $y\in f^{-m}(\bOmega)$, there is some $\zeta_m\in\JBerkf^m$
such that $y\in D(\zeta_m, \mu_m(\zeta_m)\nu(\zeta_m))$.
For any integer $\ell$ with $n<\ell\leq m$, we have $\zeta_m \in \JBerkf^{\ell}$.
By Lemma~\ref{lem:KL3new}(b), we also have
$y \in D(\zeta_m, \mu_{\ell}(\zeta_m)\nu(\zeta_m))$.
Thus, it follows from the third bullet point of Step 2 that
\[
    \big|h_{\ell}(y) - h_{\ell-1}(y)\big| < \mu_{\ell}(\zeta_m)\nu(\zeta_m). 
\]
On the other hand, it follows from our assumption~\eqref{eq:diffparts}
that $|x - y| \ge \mu_{n+1}(\zeta_m) \nu(\zeta_m)$. Therefore,
\begin{align*}
    \big|h_m(y) - h_{n}(y)\big| 
    &\le \max \big\{ \big|h_{\ell}(y) - h_{\ell-1}(y)\big| \big\}
    < \max \{ \mu_{\ell}(\zeta_m) \nu(\zeta_m) \} \\
    &= \mu_{n+1}(\zeta_m) \nu(\zeta_m)
    \le |x-y|
    = \big|h_n(x) - h_n(y)\big|        
\end{align*}
where the two maxima are over $\ell\in\{n+1,\ldots,m\}$,
and where the first equality is by Lemma~\ref{lem:KL3new}(b).
Hence,
\begin{equation}
\label{eq:Hm}
\big| H_m(x) - H_m(y) \big| = \big| h_n(x) - h_m(y) \big| = \big| h_n(x) - h_n(y) \big| = |x-y|.
\end{equation}
If $y\in\JBerkf$, we obtain $|h(x)-h(y)|=|x-y|$ by taking the limit as $m\to\infty$
in~\eqref{eq:Hm}.
Otherwise, we obtain $|h(x)-h(y)|=|x-y|$ by choosing $m$ in~\eqref{eq:Hm} to be the
largest integer for which $y\in f^{-m}(\bOmega)$.

Next, we claim that 
\begin{equation}
\label{eq:funcl}
h\big(f(\zeta) \big) = g\big( h(\zeta) \big) \quad\text{for all } \zeta\in f^{-1}(\bOmega).
\end{equation}
To see this, suppose first that $\zeta\in f^{-n}(\bOmega)\smallsetminus f^{-(n+1)}(\bOmega)$
for some $n\in\NN$. Then $h(\zeta)=h_n(\zeta)$, and $h(f(\zeta))=h_{n-1}(f(\zeta))$.
Hence, by the construction of $h_n$ in Step~1, we have
\[ g\big(h(\zeta)\big) = g\big( h_n (\zeta) \big) = h_{n-1}\big(f(\zeta) \big) = h\big(f(\zeta)\big). \]
The only other possibility is that $\zeta\in\JBerkf\cap\Cv$,
in which case $\zeta,f(\zeta)\in f^{-n}(\bOmega)$ for all $n\in\NN_0$. Therefore,
\[ g\big(h(\zeta)\big) = g\Big( \lim_{n\to\infty} h_n (\zeta) \Big)
= \lim_{n\to\infty} g\big( h_n (\zeta) \big)
= \lim_{n\to\infty} h_{n-1}\big(f(\zeta) \big) = h\big(f(\zeta)\big), \]
proving our claim.

\medskip

\textbf{Step 4}. Our goal in this step is to show that $h:\PBerk\to\PBerk$ is a homeomorphism.
We already know that $h$ fixes every point of $\PBerk\smallsetminus f^{-1}(\bOmega)$
and maps $f^{-1}(\bOmega)$ bijectively onto itself.
It follows that $h:\PBerk\to\PBerk$ is bijective.
Since $\PBerk$ is a compact Hausdorff space, it suffices to show that $h^{-1}$ is continuous.

To that end, we first recall that for every $n\in\NN$ and every $\zeta\in\JBerkf^n$,
both $h_{n-1}$ and $h_n$ are power series convergent on $\DBerk(\zeta,\mu_n(\zeta)\nu(\zeta))$
with Weierstrass degree~1.
Therefore, it is immediate from inequality~\eqref{eq:hncauchy}, along with the fact that $h_{n}$ is an isometry
on the type~I points, that
\[h_{n-1}\Big(\DBerk\big(\zeta,\mu_n(\zeta)\nu(\zeta)\big)\Big)
= h_{n}\Big(\DBerk\big(\zeta,\mu_n(\zeta)\nu(\zeta)\big)\Big)
\quad\text{for all } \zeta\in\JBerkf^n.\] 
By the definition of $H_n:\PBerk\to\PBerk$ from Step~3, it follows that
\begin{equation}
\label{eq:samedisk}
H_{n-1}\Big(\DBerk\big(\zeta,\mu_n(\zeta)\nu(\zeta)\big)\Big)
= H_{n}\Big(\DBerk\big(\zeta,\mu_n(\zeta)\nu(\zeta)\big)\Big)
\quad\text{for all } \zeta\in\JBerkf^n .
\end{equation}

Second, we claim that for every $a\in\Cv$, every $r>0$, and every $n\in\NN_0$, we have
\begin{equation}
\label{eq:samedisk2}
H_n\big( \DbarBerk(a,r) \big) = \DbarBerk\big( H_n(a),r\big)
\quad\text{and}\quad H_n\big( \DBerk(a,r) \big) = \DBerk\big( H_n(a),r\big) .
\end{equation}
We prove equation~\eqref{eq:samedisk2} by induction on $n$;
it is clearly true for $n=0$, since $H_0$ is the identity map.
For $n\in\NN$, assuming equation~\eqref{eq:samedisk2} holds for $H_{n-1}$, we now show it for $H_n$.
Let $X$ be the disk $\DbarBerk(a,r)$ or $\DBerk(a,r)$.
If $X$ does not intersect $\DBerk(\zeta,\mu_n(\zeta)\nu(\zeta))$ for any $\zeta\in\JBerkf^n$,
then $X\cap f^{-n}(\bOmega)=\varnothing$ by Lemma~\ref{lem:KL4new},
so that $H_n(X)=H_{n-1}(X)$
Similarly,
if there are any points $\zeta\in\JBerkf^n$ for which $\DBerk(\zeta,\mu_n(\zeta)\nu(\zeta))\subseteq X$,
then by equation~\eqref{eq:samedisk} and the fact that $H_{n-1}$ and $H_n$ agree outside $f^{-n}(\bOmega)$, we again have $H_n(X)=H_{n-1}(X)$.
In either case, equation~\eqref{eq:samedisk2} follows immediately.
The only remaining case is that $X\subseteq\DBerk(\zeta,\mu_n(\zeta)\nu(\zeta))$ for some $\zeta\in\JBerkf^n$.
In that case, $H_n|X=h_n|X$ is a power series convergent on the disk $X$ which is an isometry on the type~I points,
and hence equation~\eqref{eq:samedisk2} holds, proving our claim.

Third, we make the same claim for $h$: that for every $a\in\Cv$ and $r>0$, we have 
\begin{equation}
\label{eq:samedisk3}
h\big( \DbarBerk(a,r) \big) = \DbarBerk\big( h(a),r\big)
\quad\text{and}\quad h\big( \DBerk(a,r) \big) = \DBerk\big( h(a),r\big) .
\end{equation}
Let $X$ be $\DbarBerk(a,r)$ or $\DBerk(a,r)$,
and let $Y$ be $\DbarBerk(h(a),r)$ or $\DBerk(h(a),r)$, respectively.
If there is some $n\in\NN_0$ such that $X\cap f^{-n}(\bOmega)=\varnothing$,
then $h(X)=H_n(X)$, and we are done by equation~\eqref{eq:samedisk2}.
Otherwise, by Lemma~\ref{lem:KL4new},
for each $n\in\NN_0$, there is some $\zeta_n\in\JBerkf^n$ such that
$X\cap \DBerk(\zeta_n,\mu_n(\zeta_n)\nu(\zeta_n))\neq\varnothing$.
If $X\subseteq\DBerk(\zeta_n,\mu_n(\zeta_n)\nu(\zeta_n))$ for each $n$,
then $X\subseteq\Cv$ by Lemma~\ref{lem:KL5}, contradicting the fact that
the Berkovich disk $X$ contains points of type~II, for example.

Thus, there must be some $m\in\NN_0$ such that $X\supseteq\DBerk(\zeta_m,\mu_m(\zeta_m)\nu(\zeta_m))$.
By equation~\eqref{eq:samedisk2} again, we have that $H_n(X)=Y$ for every $n\geq m$.
To prove the current claim, then, it suffices to show,
for every $\xi\in\PBerk$, that $\xi\in X$ if and only if there is some $j\geq m$ such that $h(\xi)\in H_j(X)$.

Consider an arbitrary point $\xi\in\PBerk$. If there is some $j\in\NN_0$ such that $\xi\not\in f^{-j}(\bOmega)$, 
then $h(\xi)=H_j(\xi)$ by the definitions of $h$ and $H_j$, so that $\xi\in X$ if and only if $h(\xi)\in H_j(X)$.
Otherwise, $\xi\in\bigcap_{n\in\NN_0} f^{-n}(\bOmega)$. Therefore, by Lemma~\ref{lem:KL5},
we have $\xi\in\JBerkf\cap\Cv$, with $\lim_{n\to\infty} \mu_n(\xi)=0$.
Hence, there is some $j\geq m$ such that $\mu_j(\xi)\nu(\xi)< r$.
By equation~\eqref{eq:hncauchy} and the fact that $h(\xi)=\lim_{n\to\infty} h_n(\xi)$, we have
\[ \big| h(\xi) - H_j(\xi) \big| = \big| h(\xi) - h_j(\xi) \big|
= \bigg| \sum_{n=j}^{\infty} \big( h_{n+1}(\xi) - h_n(\xi) \big) \bigg| < r. \]
Therefore, $\xi\in X$ if and only if $h(\xi)\in H_j(X)$, completing
our proof our claimed equation~\eqref{eq:samedisk3}.

We are now prepared to show that $h^{-1}$ is continuous, and hence that $h$ is a homeomorphism.
For any connected open affinoid $V\subseteq\PBerk$,
it suffices to show that $h(V)$ is also open in $\PBerk$.
Write
\[ V= \PBerk\smallsetminus \big( \DbarBerk(a_1,r_1) \cup \cdots \cup \DbarBerk(a_\ell,r_\ell) \big) \]
or
\[ V = \DBerk(b,s) \smallsetminus \big( \DbarBerk(a_1,r_1) \cup \cdots \cup \DbarBerk(a_\ell,r_\ell) \big). \]
By equation~\eqref{eq:samedisk3} and the fact that $h$ is bijective, we have
\[ h(V)= \PBerk\smallsetminus
\Big( \DbarBerk\big(h(a_1),r_1\big) \cup \cdots \cup \DbarBerk\big(h(a_\ell),r_\ell\big) \Big)\]
or
\[ h(V) = \DBerk\big(h(b),s\big) \smallsetminus
\Big( \DbarBerk\big(h(a_1),r_1\big) \cup \cdots \cup \DbarBerk\big(h(a_\ell),r_\ell\big) \Big), \]
respectively. Either way, $h(V)$ is a connected open affinoid,
completing our proof that $h$ is a homeomorphism.

\medskip

\textbf{Step 5}.
We have shown that $h:\PBerk\to\PBerk$ is a homeomorphism,
mapping $\Cv$ bijectively and isometrically onto itself, 
and satisfying the conjugacy formula~\eqref{eq:funcl} on $f^{-1}(\bOmega)$.
Moreover, $f^{-1}(\bOmega)=g^{-1}(\Omega)$ by the final statement of Lemma~\ref{lem:inj}.
We now extend the conjugacy to $f^{-1}(\bOmega)\cup\JBerkf$,
and we show that $h(\JBerkf)=\JBerkg$.

To this end, we first claim that
\begin{equation}
\label{eq:fgbig}
f(\zeta)=g(\zeta) \quad\text{for all } \zeta\in f^{-1}\big(\DbarBerk(0,1)\big)
\text{ with } \diam\big(f(\zeta)\big)> t \delta^2\epsilon 
\end{equation}
where $0<t<1$ is the constant we fixed at the start of Step~1.
To see this, consider an arbitrary such point $\zeta$.
The subset $f^{-1}(\Dbar(0,1))$ of type~I points is dense in $f^{-1}(\DbarBerk(0,1))$,
whence there is a sequence $\{x_i\}_{i=0}^{\infty}\subseteq f^{-1}(\Dbar(0,1))$
such that $\lim_{i\to\infty} x_i=\zeta$. For each such type~I point $x_i$, we have
$|g(x_i)-f(x_i)| < t \delta^2\epsilon$, by Lemma~\ref{lem:nbrhd}(b). Therefore, 
\[ \big\| g - f \big\|_{\zeta} = \lim_{i\to\infty} \big\|g-f\big\|_{x_i} = \lim_{i\to\infty} \big|g(x_i)-f(x_i)\big|
\leq t\delta^2\epsilon < \diam\big( f(\zeta)\big). \] 
Thus, for any $a\in\Cv$, we have
\[ \big\| f(z)-a \big\|_\zeta \geq \diam\big( f(\zeta) \big) > \big\|g-f\big\|_{\zeta}, \]
and hence
\[ \| z-a\|_{g(\zeta)} = \big\| g(z) - a \big\|_{\zeta}
= \big\| \big( g(z) - f(z) \big)+ \big( f(z) - a \big) \big\|_{\zeta} = \big\| f(z)-a \|_\zeta = \|z-a\|_{f(\zeta)} .\]
Since this is true for all $a\in\Cv$, we have $f(\zeta)=g(\zeta)$ by \cite[Lemma~15.2(d)]{BenBook},
proving our claim.

Consider an arbitrary point $\zeta\in\JBerkf \smallsetminus f^{-1}(\bOmega)$. Then 
$f(\zeta)\in\JBerkf\smallsetminus\bOmega$, and in particular $f(\zeta)\in\JBerkf^+$. Hence,
\[ \diam\big( f(\zeta) \big) = \sphdiam\big( f(\zeta) \big) \geq \nu\big( f(\zeta) \big) > t\delta^2\epsilon ,\] 
where the equality is because $\JBerkf\subseteq\DbarBerk(0,1)$, the first inequality
is because $f(\zeta)\in\JBerkf^+$, and the second is by Lemma~\ref{lem:KL1new}(a).
Therefore, by the claim of~\eqref{eq:fgbig}, we have
\begin{equation}
\label{eq:Jplus}
f(\zeta)=g(\zeta) \quad\text{for all } \zeta\in\JBerkf \smallsetminus f^{-1}(\bOmega) .
\end{equation}

Finally, recall that $\JBerkf$ is a nonempty compact set, and hence so is its homeomorphic image $h(\JBerkf)$.
In addition, the functions $h\circ f$ and $g\circ h$ coincide on $\JBerkf$, whence
the functions $h\circ f \circ h^{-1}$ and $g$ coincide on $h(\JBerkf)$. Thus,
\[ g^{-1}\big(h(\JBerkf)\big) = h\big( f^{-1} (\JBerkf) \big) = h(\JBerkf). \]
Therefore, by \cite[Theorem~8.15(d)]{BenBook}, it follows that $h(\JBerkf)\supseteq\JBerkg$,
since $h(\JBerkf)$ is closed in $\PBerk$.
Furthermore, because of this inclusion,
we have $h^{-1}\circ g = f\circ h^{-1}$ on $\JBerkg$,
and hence we may apply the same argument to the image of the compact set $\JBerkg$
under the homeomorphism $h^{-1}$, to obtain $h^{-1}(\JBerkg)\supseteq\JBerkf$,
or equivalently, $h(\JBerkf)\subseteq\JBerkg$. Combining these two inclusions
yields the desired equality $h(\JBerkf) = \JBerkg$.

\medskip 

\textbf{Step~6}.
It remains to show that $h$ varies continuously with $g$.
More precisely, for any $q\in W$, write $h^q$ for the map $h$ constructed
in Steps~1--3 for $q$ in place of $g$,
and let $h^q_n$ denote the auxiliary functions
constructed along the way. We wish to show that
\[ \Lambda: (q,\zeta)\mapsto h^q(\zeta) \]
is a continuous function from $W\times\PBerk$ to $\PBerk$.
To this end,
given any $\xi\in\PBerk$ and an open connected affinoid $U\subseteq\PBerk$
with $h^g(\xi)\in U$, we will find an open set
$W'\subseteq W$ containing $g$ 
and an open connected affinoid $V\subseteq\PBerk$
containing $\xi$
such that $\Lambda$ maps $W'\times V$ into $U$.

Since $h=h^g:\PBerk\to\PBerk$ is a homeomorphism satisfying
equation~\eqref{eq:samedisk3}, we may write $U$ as one of the two
forms $h(V)$ given at the end of Step~4,
and then define $V:=h^{-1}(U)$ as in the same step.
Let $u:=\min\{r_1,\ldots, r_{\ell}, 1/2 \}$, so that $0<u<1$.
Define $W_u(g)\subseteq\Rat_d(\Cv)$ to be the neighborhood $W$ of $g$
given by Lemma~\ref{lem:nbrhd} for $g$ in place of $f$,
with $\gamma=u\delta^2\epsilon$.
Let $W':= W\cap W_u(g)$. For each $(q,\zeta)\in W'\times V$,
we must show $h^q(\zeta)\in U$.
By the continuity of $h^q$ (from Step~4), we may assume that $\zeta = x$
lies in $V\cap\Cv$.
Because each $x\in V\cap\Cv$ has $h(x)\in U$,
with $D(h(x),u)\subseteq U$ and $D(x,u)\subseteq V$,
it suffices to show that
\begin{equation}
\label{eq:Lcont}
\big|h^q(x) - h(x)\big| < u
\quad\text{for all } (q,x)\in W' \times \Cv .
\end{equation}

For each $q\in W'$,
consider the homeomorphism
$h^{q,g}:\PBerk\to\PBerk$ constructed according to Steps~1--3
when using $g$ in place of $f$ as the original function.
Let $G_\zeta^q:\DBerk(g(\zeta),\nu(g(\zeta)))\to \DBerk(\zeta,\mu_1(\zeta)\nu(\zeta))$
and $h_n^{q,g}:g^{-n}(\bOmega)\to q^{-n}(\bOmega)$ denote the
auxiliary functions constructed along the way.
We claim that 
\begin{equation}
\label{eq:hqginduc}
h_n^q = h_n^{q,g} \circ h_n
\quad\text{for all } n\in\NN_0 ,
\end{equation}
which we now prove by induction.
Equation~\eqref{eq:hqginduc} certainly holds for $n=0$
because all three maps are the identity map.
For $n\geq 1$, assuming the equation holds for $n-1$,
then on any disk $D(\zeta,\mu_1(\zeta)\nu(\zeta))$,
we have
\[ h_n^{q,g} \circ h_n =
G_{\zeta}^q \circ h_{n-1}^{q,g} \circ g
\circ G_{\zeta} \circ h_{n-1} \circ f
= G_{\zeta}^q \circ h_{n-1}^{q,g} \circ h_{n-1} \circ f
= G_{\zeta}^q \circ h_{n-1}^{q} \circ f = h_n^q, \]
where the second equality is because $G_{\zeta}$
is a local inverse of $g$, and the third is by
our inductive assumption.
Having proven the claim of equation~\eqref{eq:hqginduc},
it follows immediately that $h^q = h^{q,g} \circ h$.

On the other hand, by the third bullet point of Step~2
--- still with $g$, $q$, and $u$ in place of $f$, $g$, and $t$, respectively ---
we have
$|h^{q,g}_n(y)-y| < u \mu_1(\zeta) \nu(\zeta) < u$
for every $n\in\NN$, every $\zeta\in\JBerkg^n$,
and every $y\in D(\zeta,\mu_n(\zeta)\nu(\zeta))$.
The construction of $h^{q,g}$ in Step~3 therefore yields
\[ \big|h^{q,g}(y)-y\big| < u \quad\text{for all } y\in\Cv .\]
Applying this bound to $y:=h(x)$, along with the identity
$h^q=h^{q,g}\circ h$ that followed from equation~\eqref{eq:hqginduc}, we have
\[ \big| h^q(x) - h(x) \big| = 
\big| h^{q,g}\big( h(x) \big) - h(x) \big| < u
\quad\text{for all } x\in\Cv. \]
We have proven the bound~\eqref{eq:Lcont}, and hence $\Lambda$
is indeed continuous, i.e., $h$ varies continuously with $g$.
\end{proof}


\section{Examples}
\label{sec:ex}

We now present examples of rational functions satisfying the hypotheses of Theorem~\ref{MAIN}
but which are not expanding in the sense of equation~\eqref{eq:expand}.


\begin{example}
\label{EX1}
Assume the residue characteristic of $\Cv$ is $0$, and fix $c \in \Cv$ with $0 < |c| < 1$.
Define
\[ f(z) := \frac{ (z+c)(z + 1) }{ z }  = z + (c+1) + \frac{c}{z} \in \Cv(z), \]
which is a rational function of degree $d=2$.
A straightforward calculation shows that
\[ \big| f(x) - (x+1) \big| < 1 \quad \text{for all } x\in\Cv \text{ with } |x|\geq 1, \]
and therefore
\begin{equation}
\label{eq:plus1}
f \text{ maps } \DBerk(x,1) \text{ bijectively onto } \DBerk(x+1,1) \quad \text{for all } x\in\Cv \text{ with } |x|\geq 1.
\end{equation}
It follows that $\PBerk\smallsetminus\DbarBerk(0,1)\subseteq\FBerkf$, and that
$\DbarBerk(n,1)\subseteq \FBerkf$ for every positive integer $n\in\NN$.
Further simple calculations show that
\[ f\big( \DBerk(0,|c|) \big) \subseteq \PBerk\smallsetminus\DbarBerk(0,1)
\quad\text{and}\quad
f\big( \DBerk(0,1)\smallsetminus \DbarBerk(0,|c|) \big) \subseteq \DBerk(1,1). \]

Combining these facts, it follows that
\begin{equation}
\label{eq:Jex1}
\JBerkf\subseteq \big\{ \zeta\in\PBerk \, \big| \, |\zeta|=1 \text{ or } |\zeta|=|c| \big\} .
\end{equation}
Conversely, $\JBerkf$ is nonempty, and by \cite[Theorem~7.34]{BenBook}, we have $f(\DBerk(0,1))=\PBerk$.
Therefore, by~\eqref{eq:plus1}, a simple induction shows
\[ \DBerk(-n,1) \cap\JBerkf\neq\varnothing \quad\text{for all } n\in\NN .\]
Thus, $f$ is \emph{not} expanding in the sense of equation~\eqref{eq:expand},
since for any $n\in\NN$, there is some $\zeta\in\DBerk(-n,1)\cap\JBerk$,
but equation~\eqref{eq:plus1} together with Lemma~\ref{lem:injdiam}
shows that $(f^i)^{\nat}(\zeta)=1$ for all $0\leq i\leq n$.

On the other hand, we have $f(cw) = w^{-1} + 1 + c(w+1)$, and hence 
\begin{equation}
\label{eq:xtocoverx}
f \text{ maps } \DBerk(x,|c|) \text{ bijectively onto } \DBerk\bigg(\frac{c}{x}+1,1\bigg)
\quad \text{for all } x\in\Cv \text{ with } |x|=|c|, 
\end{equation} 
whence $f^{\nat}(\zeta)=|c|^{-1}$ for all $\zeta\in\PBerk$ with $|\zeta|=|c|$.
Combining this fact with equations~\eqref{eq:plus1} and~\eqref{eq:Jex1},
as well as Lemma~\ref{lem:injdiam} again, shows that
\[ \big(f^n\big)^{\nat}(\zeta)\geq 1 \quad\text{for all } \zeta\in\JBerkf \text{ and } n\in\NN . \]
That is, even though $f$ is not expanding, it satisfies the hypotheses of Theorem~\ref{MAIN} and hence is $J$-stable
in the moduli space $\Rat_2$.
\end{example}


\begin{example}
Choose an integer $m\geq 2$ such that $|m|=1$, i.e., such that
$m$ is not divisible by the residue characteristic of $\Cv$.
Fix $c \in \Cv$ with $0 < |c| < 1$.
Define
\[ f(z) := cz^{m+1} + z^{m}  = z^m (cz+1) \in \Cv[z], \]
which is a polynomial of degree $d=m+1\geq 3$. Then
\[ |f(x)| = |c| |x|^{m+1} > |x| \quad \text{for all } x\in\Cv \text{ with } |x|>|c|^{-1} \]
and
\[ |f(x)| = |x|^{m} < |x| \quad \text{for all } x\in\Cv \text{ with } |x|<1. \]
It follows that $f$ maps both $\PBerk\smallsetminus\DbarBerk(0,|c|^{-1})$ and $\DBerk(0,1)$
into themselves, and hence
\[ \DBerk(0,1) \cup \big( \PBerk\smallsetminus\DbarBerk(0,|c|^{-1}) \big) \subseteq\FBerkf. \]
Furthermore, it is not difficult to check that
\begin{equation}
\label{eq:invimf}
f^{-1}\big(\Dbar(0,|c|^{-1}) \big) \subseteq \Dbar(0,|c|^{-1/m}) \cup \Dbar(-c^{-1}, |c|^{m-2} ) 
\end{equation}
by writing $f(z)=cz^m(z+c^{-1})$. (In fact, we have equality in \eqref{eq:invimf}.)

Therefore, we have $\JBerkf\subseteq X\cup Y$, where
\[ X:= \big\{ \zeta\in\PBerk \, \big| \, 1\leq |\zeta|\leq |c|^{-1/m} \big\}
\quad\text{and}\quad Y:=\DbarBerk(-c^{-1}, |c|^{m-2}) .\]
For any $x\in X\cap\Cv$, writing $f$ as a power series centered at $x$,
it is straightforward to check that $f$ maps $D(x,|x|)$ bijectively onto $D(x^m,|x|^m)$.
Thus, (the proof of) Lemma~\ref{lem:injdiam} shows that for any $\zeta\in X$, we have
\[ f^{\nat}(\zeta) = \frac{|\zeta|^m}{|\zeta|} \cdot \frac{\max\{1,|f(\zeta)|^2\} }{\max\{ 1, |\zeta|^2 \} }
\geq \frac{|f(\zeta)|^2}{|\zeta|^2} . \]
Similarly, because $f$ maps the disk $Y$ (of diameter $|c|^{m-2}$) bijectively onto $\DbarBerk(0, |c|^{-1})$,
Lemma~\ref{lem:injdiam} shows that for any $\zeta\in Y$ with $f(\zeta)\in X\cup Y$, we have
\[ f^{\nat}(\zeta) = \frac{ |c|^{-1} }{ |c|^{m-2}} \cdot \frac{\max\{1,|f(\zeta)|^2\} }{\max\{ 1, |\zeta|^2 \} }
= |c|^{1-m} \frac{|f(\zeta)|^2}{|\zeta|^2} \geq \frac{|f(\zeta)|^2}{|\zeta|^2} . \]
Combining these two bounds, and using the fact that $\JBerkf\subseteq X\cup Y$,
we have
\[ \big( f^n \big)^{\nat}(\zeta) \geq \frac{|f^n(\zeta)|^2}{|\zeta|^2} \geq |c|^2
\quad\text{for all } \zeta\in\JBerkf \text{ and } n\in\NN .\]
Thus, $f$ satisfies the hypotheses of Theorem~\ref{MAIN} and hence is $J$-stable
in the moduli space $\Rat_{m+1}$.

On the other hand, the Newton polygon of the equation $f(z)-z=0$ reveals that $f$ has
a fixed point $a_0\in\Cv$ with $|a_0|=|c|^{-1}$. By inclusion~\eqref{eq:invimf},
we must have $a_0\in\Dbar(-c^{-1},|c|^{m-2})$. Since
\[ f'(z) = (m+1)c z^{m} + m z^{m-1} = z^{m-1} \big((m+1)c z + m \big), \]
we have $|f'(a_0)|=|c|^{1-m} > 1$, and hence $a_0$ is repelling and therefore lies in $\JBerkf$.

For each $b\in\Cv$ with $1<|b| < |c|^{-m}$, the Newton polygon of the equation $f(z)-b=0$ shows 
that $b$ has $m$ preimages $\alpha_1,\ldots,\alpha_m$ with $|\alpha_i|=|b|^{1/m}$.
Applying this fact inductively starting with $b=a_0$, and choosing only one such
preimage each time, there is an infinite sequence $\{a_n\}_{n=0}^{\infty}$ in $\Cv$ with
\[ |a_n|=|c|^{-1/m^n} \quad \text{and}\quad f(a_n)=a_{n-1} \quad\text{for all } n\in\NN .\]
Each point $a_n$ eventually maps to $a_0$ and hence lies in $\JBerkf$, with
\[ f^{\nat}(a_n) = | f'(a_n) | \cdot \frac{\max\{1,|f(a_n)|^2 \} }{\max\{1,|a_n|^2 \} }
= |a_n|^{m-1} \cdot \frac{ |a_n|^{2m} }{|a_n|^2} = |c|^{-3(m-1)/m^n} .\]
Thus,
\[ \big( f^n \big)^{\nat}(a_n) = \prod_{i=0}^{n-1} f^{\nat}(a_n) = |c|^{-e}, \]
where
\[ e = \frac{3(m-1)}{m^n} + \frac{3(m-1)}{m^{n-1}} + \cdots +  \frac{3(m-1)}{m}
= 3 \bigg( 1 - \frac{1}{m^n} \bigg) < 3.\]

Hence, $(f^n)^{\nat}(a_n)< |c|^{-3}$ for every $n\geq 1$,
and as in Example~\ref{EX1},
$f$ is \emph{not} expanding in the sense of equation~\eqref{eq:expand}.
\end{example}

\begin{remark}
\label{rem:expand}
Motivated by condition~\eqref{eq:expand}, let us call
a rational function $f : \PBerk \to \PBerk$ \emph{uniformly expanding} on its Julia set if there exist $c > 0$ and $\lambda > 1$ such that for any $n \in \NN$ and $\zeta \in \JBerkf$,
   $\sd{(f^n)}{\zeta} \ge c \lambda^n.$

Any uniformly expanding rational functions clearly satisfies the assumption of Theorem~\ref{MAIN} and hence is $J$-stable in the moduli space $\Rat_d$.
However, although this condition is appropriate in complex dynamics, 
the above examples show that uniform expansion is too restrictive
a condition in the non-archimedean setting.

In fact, any uniformly expanding rational function has
Julia set consisting only of type~I points, as we now prove.
Suppose there is $\zeta \in \JBerkf \cap \HBerk$.
Then by Proposition~\ref{prop:sphcomp}, we have
\[
    \frac{\sphdiam\big(f^n(\zeta)\big)}{\spd{\zeta}}
    = \sd{(f^n)}{\zeta} 
    \ge c \lambda^n
\]
for any $n \in \NN_0$.
Therefore,
\[
    \lim_{n \to \infty} \sphdiam\big(f^n(\zeta)\big) 
    \ge \spd{\zeta} \cdot \lim_{n \to \infty} c \lambda^n
    = \infty,
\]
contradicting the fact that $\spd{\xi} \in [0, 1]$ for any $\xi \in \PBerk$.
\end{remark}

\begin{remark} 
\label{rem:expandingrate}
In light of Remark~\ref{rem:expand}, one may ask whether
the bounded contraction hypothesis of Theorem~\ref{MAIN}
implies non-uniform exponential expansion
$(f^n)^{\nat}(\zeta)\geq c\lambda^n$ for type~I Julia
points $\zeta$, i.e., with $c$ and $\lambda$ depending
on $\zeta$. The answer is no, as we now illustrate
by revisiting the map $f(z)=(z+c)(z+1)/z$ of Example~\ref{EX1}.

Let $\{ N_i \}_{i = 1}^{\infty} \subseteq \mathbb{N}$ be any
sequence of positive integers, and define $\{M_i\}_{i=0}^{\infty}$
by $M_i:=N_1 + \cdots + N_i$. We now construct a nested sequence
$D_0\supsetneq D_1 \supsetneq D_2\supsetneq\cdots$ 
of Berkovich open disks satisfying 
\begin{itemize}
    \item $D_i$ is of the form $D_i=\DBerk(a_i,|c|^i)$
    for each $i\in\NN_0$,
    \item $f^{M_i}$ maps $D_i$ bijectively onto $D_{0}$
    for each $i\in\NN_0$, and
    \item $(f^{M_i})^{\nat}(\zeta) = |c|^{-i}$ for any
    $\zeta \in D_i$ and each $i\in\NN_0$.
\end{itemize}
To this end, we first define $D_0 := \DBerk(0, 1)$,
which clearly satisfies the above bullet points.
Proceeding inductively, having already constructed the disk
$D_{i-1}=\DBerk(a_{i-1},|c|^{i-1})$,
observe by equation~\eqref{eq:plus1} that
$f^{N_i-1}$ maps $\DBerk(1-N_i,1)$ bijectively onto $D_0$.
Moreover, by equation~\eqref{eq:xtocoverx}, $f$ maps
$\DBerk(x,|c|)\subsetneq D_0$ bijectively onto 
$\DBerk(1-N_i,1)$, where $x=-c/N_i$.
Thus, $f^{N_i}$ maps $\DBerk(x,|c|)$ bijectively onto $D_0$,
and we have $(f^{N_i})^{\nat}(\zeta)=|c|^{-1}$ for all $\zeta\in\DBerk(x,|c|)$,
by Lemma~\ref{lem:injdiam}.

Because $\DBerk(x,|c|)\subsetneq D_0$,
and $f^{M_{i-1}}$ maps $D_{i-1}$ bijectively onto $D_0$,
it follows from \cite[Proposition~3.20]{BenBook} that
there is an open disk $D_i\subseteq D_{i-1}$
such that $f^{M_i}=f^{N_i}\circ f^{M_{i-1}}$
maps $D_i$ onto $D_0$.
By Proposition~\ref{prop:chainrule} and our inductive assumptions, we have
$(f^{M_i})^{\nat}(\zeta)=|c|^{-i}$ for all $\zeta\in D_i$,
and by \cite[Proposition~3.20]{BenBook} again,
it follows that $D_i=\DBerk(a_i,|c|^i)\subsetneq D_{i-1}$
for some type~I point $a_i\in D_{i-1}$.
We have verified the bullet points above for $D_i$,
completing our inductive construction.

Because the radii of the disks
$D_0\supsetneq D_1 \supsetneq D_2\supsetneq\cdots$
decrease to $0$, their intersection is a single type~I point
$b\in\Cv$.
Any open set $U$ containing $b$ contains the disk $D_i$ for some $i$,
and by the bullet points, we have
$f^{1+M_i}(U)\supseteq f^{1+M_i}(D_i) =f(D_0)= \PBerk$.
(Recall that we verified the last equality in Example~\ref{EX1}.)
Thus, $b$ is a type~I point in $\JBerkf$.

However, if we had chosen the sequence
$\{ N_i \}_{i = 1}^{\infty}$ to increase very fast,
then the spherical derivatives $(f^n)^{\nat}(b)$ increase slowly.
For example, 
by choosing $N := \lceil|c|^{-1}\rceil \ge 2$, $N_1 := N$, 
and $N_{i+1} := N^{i + 1} - N^{i}$ for each $i \ge 2$, 
we obtain
\[ (f^{n})^{\nat}(b) \leq
    |c|^{- \lfloor\log_N(n) + 1 \rfloor} \leq |c|^{-1} n \]
for any $n \in \NN$, yielding only linear rather than exponential growth.
\end{remark}

\noindent
\textbf{Acknowledgments}.
The first author gratefully acknowledges the support of NSF grants DMS-150176 and DMS-2101925.
The second author was supported by JSPS KAKENHI Grant Number 16J01139
and the JSPS Postdoctoral Fellowship for Foreign Researchers.
This research started during the second author's stay in Amherst College.
The authors thank Laura DeMarco for helpful discussions.
We also thank the referees for their careful reading
of the paper and their suggestions, including
noting the continuous variation of the conjugacy $h$
in Theorem~\ref{MAIN}, and the questions that inspired
Remarks~\ref{rem:noperiodic} and \ref{rem:expandingrate}. 


\bibliographystyle{amsalpha}


\end{document}